\documentclass{article}
\usepackage{bm}
\usepackage{tikz,mathpazo}
\usepackage{pgf}
\usepackage[top=2.5cm, bottom=2.5cm, left=3cm, right=3cm]{geometry}   
\usepackage{indentfirst}
\usepackage{graphicx}
\usepackage{graphics}
\usepackage[toc,page,title,titletoc,header]{appendix}
\usepackage{bm}
\usepackage{bbm}
\usepackage{listings}  
\usepackage{amsmath}
\usepackage{setspace} 
\usepackage{indentfirst}
\usepackage{caption}
\usepackage{multirow} 
\usepackage{lipsum,multicol}
\usepackage{pdfpages}
\usepackage{float}
\usepackage{amsthm}
\usepackage{pdfpages}
\usepackage{fancyhdr}
\usepackage{abstract}
\usepackage{amssymb}
\usepackage{latexsym}
\usepackage{verbatim}
\usepackage[numbers]{natbib} 
\usepackage{booktabs}
\allowdisplaybreaks[4]

\newcommand{\inner}[1]{\left\langle #1 \right\rangle}
\newcommand{\norm}[1]{\left\Vert #1\right\Vert}
\newcommand{\bb}[1]{\mathbb{#1}}

\newcommand{\conv}[0]{\mathrm{conv}\,}%convex hull
\newcommand{\ri}[0]{\mathrm{ri}\,} %relative interior
\newcommand{\X}{{ \ca{X} }}

\newcommand{\ca}[1]{\mathcal{#1}}

\newcommand{\tp}{^\top}
\newcommand{\A}{\ca{A}}

\newcommand{\xk}{{x_{k} }}
\newcommand{\yk}{{y_{k} }}

\newcommand{\zk}{{z_{k} }}
\newcommand{\xkp}{{x_{k+1} }}

\newcommand{\ykp}{{y_{k+1} }}

\newcommand{\Rn}{\mathbb{R}^n}
\newcommand{\Rp}{\mathbb{R}^p}
\newcommand{\Rm}{\mathbb{R}^m}

\newcommand{\Y}{\ca{Y}}

\newcommand{\K}{\ca{K}}

\usepackage{hyperref}
\hypersetup{
    colorlinks=true,
    linkcolor=blue,
    filecolor=magenta,      
    urlcolor=blue,
    citecolor=blue,
    }

\newtheorem{theo}{Theorem}[section]
\newtheorem{lem}[theo]{Lemma}
\newtheorem{prop}[theo]{Proposition}
\newtheorem{example}[theo]{Example}

\newtheorem{defin}[theo]{Definition}

\newtheorem{assumpt}[theo]{Assumption}

\usepackage[figuresright]{rotating}
\usepackage{pdflscape}

\DeclareMathOperator*{\argmin}{arg\,min}
\DeclareMathOperator*{\argmax}{arg\,max}

\usepackage{caption}
\usepackage{algorithm}
\usepackage{algpseudocode}
\usepackage{longtable}
\usepackage{appendix}
\usepackage{tikz-cd}

\numberwithin{equation}{section}

\usepackage{array}

\title{A Minimization Approach for Minimax Optimization with Coupled Constraints\thanks{All authors contributed equally to this work and are listed in alphabetical order.}}
\author{Xiaoyin Hu\thanks{School of Computer and Computing Science, Hangzhou City University, Hangzhou, 310015, China.  (hxy@amss.ac.cn).}, ~
Kim-Chuan Toh\thanks{Department of Mathematics, and Institute of Operations Research and Analytics, National University of Singapore, Singapore 119076. (mattohkc@nus.edu.sg).}, ~
Shiwei Wang\thanks{Institute of Operational Research and Analytics, National University of Singapore, Singapore. (wangshiwei@amss.ac.cn).}, ~
Nachuan Xiao\thanks{Institute of Operational Research and Analytics, National University of Singapore, Singapore. (xnc@lsec.cc.ac.cn).}}

\begin{document}
\maketitle

\begin{abstract}
    In this paper, we focus on the nonconvex-strongly-concave minimax optimization problem (MCC), where the inner maximization subproblem contains constraints that couple the primal variable of the outer minimization problem. We prove that by introducing the dual variable of the inner maximization subproblem, (MCC) has the same first-order minimax points as a nonconvex-strongly-concave minimax optimization problem without coupled constraints (MOL). We then extend our focus to a class of nonconvex-strongly-concave minimax optimization problems (MM) that generalize (MOL). By performing the partial forward-backward envelope to the primal variable of the inner maximization subproblem, we propose a minimization problem (MMPen), where its objective function is explicitly formulated. We prove that the first-order stationary points of (MMPen) coincide with the first-order minimax points of (MM). Therefore, various efficient minimization methods and their convergence guarantees can be directly employed to solve (MM), hence solving (MCC) through (MOL). Preliminary numerical experiments demonstrate the great potential of our proposed approach. 
\end{abstract}
	
\section{Introduction}

In this paper, we consider the following minimax optimization problem with coupled constraints, 
\begin{equation}
    \label{Prob_Con_Minmax}
    \tag{MCC}
    \min_{x \in \X} \left( \max_{y \in \Y, ~ c(x, y) \in \K} ~  \phi(x, y) :=  g(x, y) + r_1(x) \right), 
\end{equation}
where $\X$ is a closed subset of $\Rn$ (not necessarily convex), $\Y$ is a closed convex subset of $\Rp$, and $\K$ is a closed convex cone in $\Rm$. We present the basic assumptions on \eqref{Prob_Con_Minmax} in the following.  
\begin{assumpt}
    \label{Assumption_f}
    \begin{enumerate}
        \item $g: \Rn \times \Rp \to \bb{R}$ is locally Lipschitz smooth over $\X \times \Y$.
        \item The functions $r_1: \Rn \to \bb{R}$ is locally Lipschitz continuous. 
        \item For any $x \in \X$, the function $y \mapsto  g(x, y)$ is $\mu$-strongly concave, where $\mu$ is independent of $x$. 
        \item The constraint mapping $c: \X \times \Y \to \bb{R}^m$ is locally Lipschitz smooth. Moreover, for any $x \in \X$, we assume that 
        \begin{enumerate}
            \item for any $\lambda \in \K^{\circ}$, the function $y \mapsto \inner{\lambda, c(x, y)}$ is convex over $\Y$;
            \item  for any $y \in \Y$ that satisfies $c(x, y) \in \K$, the following nondegeneracy condition for the inner maximization subproblem holds, 
            \begin{equation*}
                \bb{R}^{m+p}  = 
                \left[
                \begin{smallmatrix}
                    \nabla_y c(x, y)\tp\\
                    I_p\\
                \end{smallmatrix}
                \right] \Rp 
                + 
                \mathrm{lin}\left(
                \left[
                \begin{smallmatrix}
                    \ca{T}_{\K}(c(x, y))\\
                    \ca{T}_{\Y}(y)\\
                \end{smallmatrix}
                \right]
                \right);
            \end{equation*}
            \item  there exists $y \in \Y$ such that $y \in \mathrm{ri}(\Y)$ and $c(x, y) \in \mathrm{ri}(\K)$.
        \end{enumerate}
        % \item For any $x \in \X$,
    \end{enumerate}
\end{assumpt}
Here $\ca{T}_{\K}(\cdot)$ and $\ca{T}_{\Y}(\cdot)$ are the tangent cones of $\K$ and $\Y$, respectively. Moreover, $\nabla_y c: \X \times \Y \to \bb{R}^{p \times m}$ is the partial Jacobian of $c$ with respect to its $y$-variable. Furthermore, for any given subset $\tilde{\K}$, the notation $\K^{\circ}$, $\mathrm{lin}\left( \tilde{\K} \right)$, and $\ri(\tilde{\K})$ refer to the polar cone of $\tilde{\K}$, the largest subspace contained in $\tilde{\K}$, and the relative interior of $\tilde{\K}$, respectively.  In addition, for any closed convex subset $\Omega$, we say that a continuous function $f: \Omega \to \bb{R}$ is $\mu$-strongly concave, if for any $x,y \in \Omega$, it holds that $f(y) \leq f(x) + \inner{v, y-x} - \frac{\mu}{2} \norm{y-x}^2$ for any $v \in \partial f(x)$, where $\partial f(x)$ refers to the Clarke subdifferential of $f$ at $x$.

When there is no coupled constraint in \eqref{Prob_Con_Minmax} (i.e., $c(x, y) = 0$ holds for all $(x, y) \in \X\times \Y$), the minimax optimization problem \eqref{Prob_Con_Minmax} reduces to the nonconvex-strongly-concave minimax optimization problem. Therefore, the minimax optimization problem that takes the form of \eqref{Prob_Con_Minmax} 
has wide applications in various areas, such as adversarial training tasks \cite{goodfellow2014explaining,madry2018towards,sinha2018certifiable,tsaknakis2023minimax,hu2022improved},  bilevel optimization with lower-level constraints \cite{lu2023solving}, generative adversarial networks \cite{gidel2018a,goodfellow2014generative,sanjabi2018convergence}, etc. Specifically, \cite{tsaknakis2023minimax,lu2023first} demonstrate how the minimax optimization model \eqref{Prob_Con_Minmax}, which incorporates non-trivial coupled constraints, is applied to tackle resource allocation and network flow issues in the context of adversarial attacks.

For the minimax optimization problem \eqref{Prob_Con_Minmax}, we denote 
\begin{equation}
    \label{Eq_intro_defin_Phi}
    \Phi(x) := \max_{y \in \Y, ~ c(x, y)\in \K} \phi(x,  y).
\end{equation}
Then  \eqref{Prob_Con_Minmax} is equivalent to the following minimization problem over $\X$:
\begin{equation}
    \label{Eq_intro_F_min}
    \min_{x \in \X} ~ \Phi(x).
\end{equation}
Therefore, the first-order, local and global {\it minimax points} of \eqref{Prob_Con_Minmax} can viewed as the first-order, local and global stationary points of the minimization problem  \eqref{Eq_intro_F_min}, respectively (see Section 2 for detailed definitions). However, as the inner maximization subproblem of \eqref{Prob_Con_Minmax} involves coupled constraints, the subdifferentials of $\Phi$ usually do not have explicit formulations, making the direct minimization of $\Phi$ over $\X$ challenging in practice. 

Motivated by the Lagrangian dual approach \cite{hestenes1969multiplier,powell1969method} and the duality techniques for minimax optimization (e.g., \cite{tsaknakis2023minimax,lu2023first}), we consider the following reformulation of \eqref{Prob_Con_Minmax},
\begin{equation}
    \label{Eq_intro_reformulate}
    \min_{x \in \X} \left( \max_{y \in \Y} \min_{\lambda \in \K^{\circ}} ~  L(x, \lambda, y) \right).
\end{equation}
Here $L(x, \lambda, y)$ is the Lagrangian function for the inner maximization subproblem of \eqref{Prob_Con_Minmax}, which is defined as 
\begin{equation}
    L(x, \lambda, y) := \phi(x, y) - \inner{\lambda, c(x, y)}. 
\end{equation}
Since $L(x, \lambda, y)$ is strongly concave with respect to the $y$-variable for any given $x \in \X$ and $\lambda \in \K^{\circ}$, the strong duality \cite[Theorem 37.3]{rockafellar1970convex} of the inner maximization subproblem of \eqref{Prob_Con_Minmax} illustrates that the min-max-min optimization problem \eqref{Eq_intro_reformulate} is equivalent to the following min-min-max optimization problem under Assumption \ref{Assumption_f}, 
\begin{equation}
    \label{Prob_Con_Minmax_Reform}
    % \tag{M\textsuperscript{3}L}
    \min_{x \in \X} \left(  \min_{\lambda \in \K^{\circ} }\max_{y \in \Y} ~  L(x, \lambda, y) \right).
\end{equation}
By merging the minimization with respect to both $x$- and $\lambda$-variables in \eqref{Prob_Con_Minmax_Reform}, we arrive at the following minimax optimization problem without coupled constraints,
\begin{equation}
    \label{Prob_Pen_Minmax}
    \tag{MOL}
    \min_{x \in \X, ~\lambda \in \K^{\circ} } \left(\max_{y \in \Y} ~  L(x, \lambda, y)\right).
\end{equation}
It is worth mentioning that without Assumption \ref{Assumption_f}, \eqref{Prob_Con_Minmax} is not equivalent to \eqref{Prob_Pen_Minmax} 
with respect to their first-order minimax points. In the following, we present an illustrative example showing that  \eqref{Prob_Pen_Minmax} may introduce spurious first-order minimax points to \eqref{Prob_Con_Minmax} in the absence of Assumption \ref{Assumption_f}(4b).
\begin{example}
    \label{Example_counter_1}
    Consider the following minimax optimization problem in $[1, 10] \times \bb{R}$,
    \begin{equation}
        \label{Eq_counter_example_0}
        \min_{1\leq x\leq 10} \left(\max_{y \in \bb{R},~ c(x, y)\leq 0} ~ - \frac{1}{2}(y-2x)^2\right) ,
    \end{equation}
    where $c(x, y) = [y-x, y - x^4]$. 
    Then it is easy to verify that $\Phi(x) = -\frac{1}{2}x^2$ (see \eqref{Eq_intro_defin_Phi} for the definition of $\Phi$), hence the only first-order minimax point of the \eqref{Eq_counter_example_0} is $(x ,y) = (10,  10)$. Moreover, for the minimax optimization problem \eqref{Eq_counter_example_0}, it is worth mentioning that the nondegeneracy condition in Assumption \ref{Assumption_f}(4b) fails at $(x, y) = (1,1)$, while all the other conditions in Assumption \ref{Assumption_f} are satisfied. 

    Next, we show that $(x, y) = (1,1)$ corresponds to a first-order minimax point of \eqref{Prob_Pen_Minmax}. When applying the formulation of \eqref{Prob_Pen_Minmax} to \eqref{Eq_counter_example_0}, we arrive at the following minimax optimization problem,
    \begin{equation}
        \label{Eq_counter_example_1}
        \min_{1\leq x\leq 10, \lambda_1, \lambda_2 \geq 0} \left(\max_{y \in \bb{R}} ~ - \frac{1}{2}(y-2x)^2 - \lambda_1(y-x) - \lambda_2(y - x^4) \right). 
    \end{equation}
    Then it is easy to verify that $(x, \lambda_1, \lambda_2, y) = (1, \frac{2}{3}, \frac{1}{3}, 1)$ is a first-order minimax point of \eqref{Eq_counter_example_1}. Therefore, we can conclude that without Assumption \ref{Assumption_f}(4b), transforming \eqref{Eq_counter_example_0} into \eqref{Eq_counter_example_1} may introduce spurious first-order minimax points. 
\end{example}

To develop efficient optimization methods for solving \eqref{Prob_Con_Minmax}, the pioneering work by \cite{tsaknakis2023minimax} proposes a double-loop method for the following reformulation of \eqref{Prob_Con_Minmax},
\begin{equation}
    \label{Eq_Intro_minimax_reform}
    \min_{\lambda \in \K^{\circ}}\left(\min_{x \in \X} \max_{y \in \Y} ~  L(x, \lambda, y)\right), 
\end{equation}
under the assumptions that $c(x, y)$ is a linear function and $\phi$ is strongly convex with respect to the $x$-variable. Consequently, the validity of Assumption \ref{Assumption_f}(4b) can be inferred from their adopted conditions. Their proposed method requires multiple steps to achieve an approximate minimax point of the inner minimax optimization problem in \eqref{Eq_Intro_minimax_reform}, followed by a gradient-descent step to update the $\lambda$-variable. However, balancing the computational cost of the inner loop and the overall performance of the double-loop method often presents a significant challenge. Furthermore, the assumptions regarding the strong convexity of $\phi$ and the linearity of the coupled constraints limit its applicability to various real-world scenarios. Recently, another work by \cite{lu2023first} develops an efficient Lagrangian-based method to solve \eqref{Prob_Con_Minmax} through \eqref{Prob_Pen_Minmax}, where the non-degeneracy condition in Assumption \ref{Assumption_f}(4b) is absent. Consequently, their proposed method only guarantees convergence to first-order minimax points of \eqref{Prob_Pen_Minmax}, which may not necessarily be first-order minimax points of \eqref{Prob_Con_Minmax}. How to develop efficient optimization methods for \eqref{Prob_Con_Minmax} with guaranteed convergence to its first-order minimax points remains unexplored.

\subsection{Related works}

The minimax optimization problem \eqref{Prob_Pen_Minmax} falls into the category of nonconvex strongly-concave minimax optimization problems, which generally can be formulated as
\begin{equation}
    \label{Prob_Minmax}
    \tag{MM}
    \min_{x \in \X} \left(\max_{y \in \Y} ~  h(x, y) := f(x, y) + r_1(x) - r_2(y) \right).
\end{equation}
Here $f: \Rn \times \Rp \to \bb{R}$ is locally Lipschitz smooth and strongly concave with respect to the $y$-variable for any fixed $x \in \X$. Moreover, $r_2: \Rp \to \bb{R}$ is locally Lipschitz continuous and convex. 

To transform nonsmooth objective functions into their differentiable counterparts, the Moreau envelope is a powerful tool for nonsmooth nonconvex minimization problems \cite{zhang2018convergence,davis2020stochastic}. Moreover, several existing works \cite{attouch1983convergence,grimmer2023landscape} generalize the Moreau envelope from minimization problems to minimax optimization problems. When applied to \eqref{Prob_Minmax}, the Moreau envelope yields the following nonconvex-strongly-concave minimax optimization problem \cite{grimmer2023landscape},
\begin{equation}
    \label{Eq_Intro_Moreau_Minimax}
    \min_{x \in \X} \left(\max_{y \in \Rn} ~ \Psi_{M, \eta}(x, y)\right), 
\end{equation}
where $\Psi_{M, \eta}(x, y) := \max_{v \in \Y} h(x, v) - \frac{1}{2\eta}\norm{v - y}^2$ for a prefixed $\eta > 0$. However, solving the proximal point subproblem just described
may be intractable in practice. Therefore, the objective function $\Psi_{M, \eta}$ is implicitly formulated, in the sense that computing the exact objective functions and gradients of $\Psi_{M, \eta}$ may be challenging in practice. As a result, most existing works on solving \eqref{Prob_Con_Minmax} through \eqref{Eq_Intro_Moreau_Minimax} focus on proximal point methods \cite{rockafellar1976monotone,nemirovski2004prox,tseng1995linear,daskalakis2018limit}. These methods, however, involve intractable proximal subproblems and may exhibit inefficient performance in practical scenarios.

The forward-backward envelope \cite{stella2017forward,themelis2018forward} is a parallel tool to the Moreau envelope, developed to transform nonsmooth nonconvex minimization problems into differentiable nonconvex ones. As demonstrated in \cite{stella2017forward,themelis2018forward}, the forward-backward envelope yields an explicitly formulated objective function, in the sense that the function values and the exact gradients of the forward-backward envelope can be computed directly based on solving the proximal gradient subproblem. However, all the existing works on forward-backward envelope focus on minimization optimization problems, with no existing work investigating how to extend the forward-backward envelope to minimax optimization problems.

For solving the nonconvex-strongly-concave optimization problem \eqref{Prob_Minmax}, most existing optimization methods \cite{jin2020local,hassan2018non,lin2020gradient,lu2020hybrid,yang2022nest,luo2022finding,li2022nonsmooth,yang2024data,zheng2024universal} are developed by extending the optimization methods for minimization problems to minimax ones, based on the descent-ascent scheme. In particular, \cite{xu2023unified} develops the proximal gradient descent-ascent methods for \eqref{Prob_Minmax}, which can be viewed as an extension of proximal gradient methods for minimization problems. Moreover, \cite{luo2022finding} extends the cubic Newton method to minimax optimization problems and proposes a cubic Newton method for \eqref{Prob_Minmax}. In addition, \cite{yang2022nest} extends the AdaGrad method to minimax optimization problems, which yields the NeAda method for solving \eqref{Prob_Minmax}. Although there exist various well-established convergence results for minimization problems, the convergence guarantees of these descent-ascent methods for minimax ones should be re-established. 
On the other hand, while the recent work by \cite{cohen2024alternating} considers transforming \eqref{Prob_Minmax} into a minimization problem, their developed minimization problem \cite[Problem P]{cohen2024alternating} is implicitly formulated. As a result, their proposed framework is restricted to proximal gradient descent-ascent methods. Therefore, we are driven to ask the following question,
\begin{quote}
    Can we directly apply existing optimization methods for minimization problems and their convergence properties to \eqref{Prob_Con_Minmax}, with guaranteed convergence to its first-order minimax points?
\end{quote}

\subsection{Contributions}

We first prove that under Assumption \ref{Assumption_f}, \eqref{Prob_Con_Minmax} and \eqref{Prob_Pen_Minmax}  have the same first-order minimax points. Our results illustrate that solving the minimax problem with coupled constraints \eqref{Prob_Con_Minmax} is equivalent to solving the minimax problem \eqref{Prob_Pen_Minmax} without any coupled constraint. Moreover, as there is no coupled constraint in \eqref{Prob_Pen_Minmax}, various existing efficient methods for nonconvex-strongly-concave minimax optimization can be directly applied to solve \eqref{Prob_Con_Minmax} through \eqref{Prob_Pen_Minmax}, with guaranteed convergence to the first-order minimax points of \eqref{Prob_Con_Minmax}.

Then we focus on developing efficient methods for the nonconvex-strongly concave minimax problem \eqref{Prob_Minmax}, which generalizes \eqref{Prob_Pen_Minmax}. Motivated by \cite{themelis2018forward}, we consider the following partial forward-backward envelope (PFBE) of $f$ with respect to its $y$-variable for \eqref{Prob_Minmax},
\begin{equation}
    \label{Eq_FBE_Y}
    \tag{PFBE}
    \begin{aligned}
        \Psi_{\eta}(x, y) := \max_{v\in \Y} ~ f(x, y) + \inner{\nabla_y f(x, y), v-y} - r_2(v) - \frac{1}{2\eta} \norm{v-y}^2. 
    \end{aligned}
\end{equation}
The PFBE generalizes the forward-backward envelope  \cite{themelis2018forward} but differs in key aspects. In particular, under mild conditions, we show that for the following minimization problem, 
\begin{equation}
    \label{Prob_Pen}
    \tag{MMPen}
    \min_{x \in \X, ~y \in \Y} ~  \Gamma_{(\eta, \alpha)}(x, y) := \alpha \Psi_{\eta}(x, y) - (\alpha-1) f(x, y) + r_1(x) + (\alpha -1) r_2(y),
\end{equation}
its first-order stationary points are exactly the first-order minimax points of \eqref{Prob_Minmax} with sufficiently large but finite $\alpha \geq 1$. 

This equivalence establishes an important link between minimax optimization and minimization optimization, which implies that some descent-ascent methods for \eqref{Prob_Minmax} can be interpreted as descent methods for \eqref{Prob_Pen}. In particular, we show that the subgradient descent-ascent method \cite{cohen2024alternating} can be explained as an inexact subgradient descent method for \eqref{Prob_Pen}. Then we prove the global convergence of the subgradient descent-ascent method under nonsmooth nonconvex settings through the ordinary differential equation (ODE) approach developed by \cite{xiao2023convergence}, where $\Gamma_{(\eta, \alpha)}$ in \eqref{Prob_Pen} serves as a Lyapunov function for the corresponding differential inclusion. 

Furthermore, through the equivalence between \eqref{Prob_Minmax} and  \eqref{Prob_Pen}, various efficient methods for minimization over $\X \times \Y$ can be directly employed to solve \eqref{Prob_Minmax} with convergence guarantees. Preliminary numerical experiments on synthetic test problems illustrate that directly implementing the existing efficient solvers to \eqref{Prob_Pen} can achieve high efficiency in solving \eqref{Prob_Con_Minmax}. These results further demonstrate the great potential of performing PFBE to nonconvex-strongly-concave minimax optimization problems.

\begin{figure}[!tb]
    \footnotesize
    \caption{A brief illustration of our proposed minimization approach for minimax optimization problem with coupled constraints \eqref{Prob_Con_Minmax}.}
    \label{Figure_FBEY}
    \begin{equation*}
        \begin{tikzcd}[row sep=4em,column sep=12em,cells={nodes={draw=black}}]
            \parbox{4.5cm}{\centering Minimax problem \eqref{Prob_Con_Minmax} with coupled constraints} \arrow[d, Leftrightarrow, "\substack{\text{Lagrangian approach} \\\text{under Assumption \ref{Assumption_f}}}"] \arrow[r, dashleftarrow, "\substack{\text{Require re-establishing} \\\text{convergence results}}"]& \parbox{7cm}{\centering Specifically designed multiple-loop methods} \\
            \parbox{4.5cm}{\centering Minimax problem \eqref{Prob_Pen_Minmax} as a special case of \eqref{Prob_Minmax}} \arrow[d, Leftrightarrow, "\substack{\text{PFBE with respect to $y$}\\ \text{under Assumption \ref{Assumption_f_Minmax}} }"]  \arrow[r, Leftarrow, "\substack{\text{Direct implementation}}"] & \parbox{7cm}{\centering Subgradient descent-ascent method \eqref{Eq_SubGDA} for \eqref{Prob_Minmax}} \arrow[d, Leftarrow, "\substack{\text{Special case: inexact}\\\text{subgradient descent method}}"]\\
            \parbox{4.5cm}{\centering Minimization problem in the form of \eqref{Prob_Pen}}  \arrow[r, Leftarrow, "\substack{\text{Direct implementation}}"] & \parbox{7cm}{\centering Existing methods for minimization over $\X \times \Y$ with guaranteed convergence}
        \end{tikzcd}
    \end{equation*}
\end{figure}
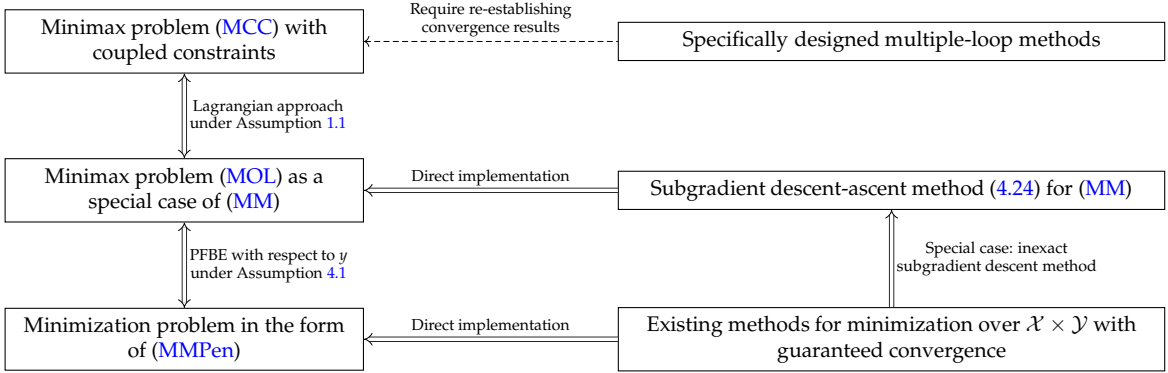

\subsection{Organization}
The outline of the rest of this paper is as follows. In Section 2, we present the notations and preliminary concepts that are necessary for the proofs in this paper. We establish the equivalence between \eqref{Prob_Con_Minmax} and \eqref{Prob_Pen_Minmax} in the aspect of first-order minimax points in Section 3. In Section 4, we focus on the general nonconvex-strongly-concave minimax problem \eqref{Prob_Minmax}, and prove its equivalence to the minimization optimization problem \eqref{Prob_Pen}. Section 5 presents illustrative numerical examples to show that \eqref{Prob_Pen} enables straightforward implementations of various existing efficient minimization approaches for solving \eqref{Prob_Con_Minmax}. We conclude the paper in the last section.

\section{Preliminaries}

\subsection{Notations}
Firstly, we denote $\Rn_+ := \{x = (x_1, \ldots, x_n) \in \Rn : x_i \geq 0 \text{ for all } i = 1, \ldots, n\}$, and $\ca{B}_\delta(x)$ as the closed ball centered at $x$ with radius $\delta$.  For a closed set $\X \subset \Rn$, we denote  $\Pi_{\X}(x):= \argmin_{y \in \X} \norm{y-x}^2$. Besides, we denote $\ca{T}_{\X}(x)$ and $\ca{N}_{\X}(x)$ as the (Clarke) tangent cone and normal cone of $\X$ at $x \in \X$, respectively \cite{davis2020stochastic}.  Moreover, for any subsets $\ca{X}, \ca{Y} \subset \Rn$, we denote $\inner{\X, \Y}:= \{\inner{x, y}: x \in \ca{X}, y\in \ca{Y} \}$,  $|\ca{X}|:= \{|x|: x \in \ca{X}\}$ and $\norm{\ca{X}} := \sup\{ \norm{w} : w\in \ca{X}\}$. We denote $\mathrm{span}(A)$ as the subspace spanned by the column vectors of matrix $A$. We say $\X \perp \Y$ if $\inner{x, y} = 0$ holds for any $x \in \X$ and any $y \in \Y$, and denote $\X^\perp$ as the orthogonal complement of $\X$.  For any $z \in \Rn$, we denote $z + \ca{X} := \{z\} + \ca{X}$ and $\inner{z , \ca{X}} := \inner{\{z\} ,\ca{X}}$. In addition, for any $\ca{Z} \subseteq \bb{R}$ and any $t \in \bb{R}$, we say $t \geq \ca{Z}$ if $t \geq z$ for all $z \in \ca{Z}$.  Furthermore, for any given locally Lipschitz continuous function $f: \Rn \to \bb{R}$, the Clarke subdifferential of $f$ is defined as
    \begin{equation*}
        \partial f(x) := \conv\left( \left\{ \lim_{k\to +\infty} \nabla f(\zk): \text{$f$ is differentiable at $\{\zk\}$, and } \lim_{k\to +\infty} \zk = x   \right\}  \right).
    \end{equation*}

For the minimax optimization problem \eqref{Prob_Con_Minmax}, for any $x \in \X$, we denote 
\begin{equation}
    \label{eq:G}
        y^*(x) := \argmax_{y \in \Y,~ c(x, y) \in \K} ~ \phi(x, y), \quad \lambda^*(x) := \argmin_{\lambda \in \K^{\circ}} ~L(x, \lambda, y^*(x)), \quad G(x) := g(x, y^*(x)).
\end{equation}
Moreover, we denote  
\begin{equation}
    \label{eq:LHG}
    \begin{aligned}
        &L_{P}(x, \lambda, y):= g(x, y) - \inner{\lambda, c(x, y)}, \quad  
        H_T(x, \lambda) := \max_{y \in \Y} ~L_{P}(x, \lambda, y),
        \quad y^*_T(x, \lambda) := \argmax_{y \in \Y} ~L_{P}(x, \lambda, y), 
        \\
        &
         G_T(x) := \min_{\lambda \in \K^{\circ}} H_T(x, \lambda) \;=\;
         \min_{\lambda \in \K^{\circ}} \max_{y \in \Y}  ~ L_{P}(x, \lambda, y), \quad \lambda^*_T(x):=  \argmin_{\lambda \in \K^{\circ} }~ H_T(x, \lambda).
    \end{aligned}
\end{equation}

% \begin{eqnarray}\label{eq:G}
%  &&   y^*(x) := \argmax_{y \in \Y,\, \red{c(x, y)\in {\cal K}}} ~ \phi(x, y), \quad \lambda^*(x) := \argmin_{\lambda \in \K^{\circ}} ~L(x, \lambda, y^*(x)), \quad G(x) := g(x, y^*(x)) 
%     \\
% &&    \begin{aligned}
%         L_{P}(x, \lambda, y):={}& g(x, y) - \inner{\lambda, c(x, y)}, \; H_T(x, \lambda) := \max_{y \in \Y} ~L_{P}(x, \lambda, y), 
%         \\[5pt]
%         y^*_T(x, \lambda) :={}& \argmax_{y \in \Y} ~L_{P}(x, \lambda, y), \;\; 
%          G_T(x) := \min_{\lambda \in \K^{\circ}} \max_{y \in \Y}  ~ L_{P}(x, \lambda, y), \;\; \lambda^*_T(x):=  \argmin_{\lambda \in \K^{\circ} }~ H_T(x, \lambda).  
%     \end{aligned}
%     \label{eq:LHG}
% \end{eqnarray}

\subsection{Optimality conditions}
In this subsection, we present the optimality condition for the minimax problem \eqref{Prob_Con_Minmax}, \eqref{Prob_Minmax}, and the minimization problem \eqref{Prob_Pen}.

We first present the definitions on the global and local minimax points of \eqref{Prob_Con_Minmax} and \eqref{Prob_Pen_Minmax}, which follows from \cite{jiang2023optimality,tsaknakis2023minimax}. 
\begin{defin}
    \label{Def_Globalstat_MinimaxCon}
    We say that $(\tilde{x}, \tilde{y}) \in \X \times \Y$ is a global minimax point of  \eqref{Prob_Con_Minmax}, if 
    \begin{equation}
        \tilde{x} \in \argmin_{x \in \X} ~\Phi(x), \quad \text{and} \quad \tilde{y} \in y^*(\tilde{x}).
    \end{equation}

    Moreover, we say that $(\tilde{x}, \tilde{y}) \in \X \times \Y$ is a local minimax point of  \eqref{Prob_Con_Minmax}, if there exists $\delta > 0$ such that 
    \begin{equation}
        \tilde{x} \in \argmin_{x \in \X \cap \ca{B}_{\delta}(\tilde{x})} ~\Phi(x), \quad \text{and} \quad \tilde{y} \in y^*(\tilde{x}).  
    \end{equation}
\end{defin}

\begin{defin}
    We say that $(\tilde{x}, \tilde{\lambda}, \tilde{y}) \in \X \times \K^{\circ} \times \Y$ is a global minimax point of  \eqref{Prob_Pen_Minmax}, if
    \begin{equation}
        (\tilde{x}, \tilde{\lambda}) \in \argmin_{x \in \X, \lambda \in \K^{\circ}} ~ H_T(x, \lambda), \quad \text{and} \quad \tilde{y} \in y^*_T(\tilde{x}, \tilde{\lambda}). 
    \end{equation} 
    Moreover, we say that $(\tilde{x}, \tilde{\lambda}, \tilde{y}) \in \X \times \K^{\circ} \times \Y$ is a local minimax point of  \eqref{Prob_Con_Minmax}, if there exists $\delta > 0$ such that 
    \begin{equation}
        (\tilde{x}, \tilde{\lambda}) \in \argmin_{ (x, \lambda) \in 
        (\X \times \K^{\circ}) \cap \ca{B}_{\delta}(\tilde{x}, \tilde{\lambda})  } ~ H_T(x, \lambda), \quad  \text{and} \quad  \tilde{y} \in y^*_T(\tilde{x}, \tilde{\lambda}). 
    \end{equation} 
\end{defin}
Then, we present the definition on the first-order minimax points of \eqref{Prob_Con_Minmax}.
\begin{defin}
    \label{Def_FOSP_MinimaxCon}
    We say that $(\tilde{x}, \tilde{y}) \in \X \times \Y$ is a first-order minimax point of  \eqref{Prob_Con_Minmax}, if 
    \begin{equation}
        0\in \partial \Phi(\tilde{x}) +  \ca{N}_{\X}(\tilde{x}), \quad \text{and} \quad \tilde{y} \in y^*(\tilde{x}).
    \end{equation}
\end{defin}
It is worth mentioning that from the strong concavity of $\phi$ with respect to the $y$-variable in Assumption \ref{Assumption_f}, $y^*(x)$ is a singleton for any $x \in \X$. Therefore, in the rest of this paper, we regard $y^*(\cdot)$ as a mapping from $\Rn$ to $\Rp$.  

In addition, as the minimax problem \eqref{Prob_Pen_Minmax} does not contain the coupled constraints, its first-order minimax point can be defined as below, which is a natural extension of the optimality condition in \cite[Equation 11(23)]{rockafellar2009variational}. 
\begin{defin}
    We say that $(\tilde{x}, \tilde{\lambda}, \tilde{y}) \in \X \times \K^{\circ} \times \Y$  is a first-order minimax point of \eqref{Prob_Pen_Minmax}, if 
    \begin{equation}
        \left\{
        \begin{aligned}
            &0 \in \nabla_x  g(\tilde{x}, \tilde{y}) - \nabla_x c(\tilde{x}, \tilde{y})\tilde{\lambda} + \partial r_1(\tilde{x}) + \ca{N}_{\X}(\tilde{x}),\\
            &0 \in -\nabla_y g(\tilde{x}, \tilde{y}) + \nabla_y c(\tilde{x}, \tilde{y})\tilde{\lambda} + \ca{N}_{\Y}(\tilde{y}),\\
            &0 \in -c(\tilde{x}, \tilde{y}) + \ca{N}_{\K^{\circ}}(\tilde{\lambda}).
        \end{aligned}
        \right. 
    \end{equation}
\end{defin}

Next, we give the definition of the first-order minimax points of  \eqref{Prob_Minmax}.
\begin{defin}
    \label{Def_FOSP_MinMax}
    We say that $(\tilde{x}, \tilde{y}) \in \X \times \Y$ is a first-order minimax point of \eqref{Prob_Minmax}, if
    \begin{equation}
        \left\{
        \begin{aligned}
            & 0 \in \nabla_x f(\tilde{x}, \tilde{y}) + \partial r_1(\tilde{x}) + \ca{N}_{\X}(\tilde{x}),\\
            & 0 \in -\nabla_y f(\tilde{x}, \tilde{y}) + \partial r_2(\tilde{y}) + \ca{N}_{\Y}(\tilde{y}). 
        \end{aligned}
        \right. 
    \end{equation}
    
    Moreover, we say that $(\tilde{x}, \tilde{y}) \in \X \times \Y$ is an $\varepsilon$-first-order minimax point of \eqref{Prob_Minmax}, if
    \begin{equation}
        \left\{
        \begin{aligned}
            & \mathrm{dist} \left(0, \nabla_x f(\tilde{x}, \tilde{y}) + \partial r_1(\tilde{x}) + \ca{N}_{\X}(\tilde{x}) \right) \leq \varepsilon,\\
            & \mathrm{dist}\left(0, -\nabla_y f(\tilde{x}, \tilde{y}) + \partial r_2(\tilde{y}) + \ca{N}_{\Y}(\tilde{y}) \right) \leq \varepsilon. 
        \end{aligned}
        \right. 
    \end{equation}
\end{defin}

The optimality condition for the minimization problem of the form $\min_{x \in \X, ~y \in \Y}~ h(x, y)$ is given as follows. 
\begin{defin}
    \label{Def_FOSP_Pen}
    For any given locally Lipschitz continuous function $h: \X\times \Y \to \bb{R}$, we say that $(\tilde{x}, \tilde{y}) \in \X \times \Y$ is a first-order stationary point of the minimization problem $\min_{x \in \X, ~y \in \Y}~ h(x, y)$ if  
    \begin{equation}
        \left\{
        \begin{aligned}
            & 0 \in \partial_x h(\tilde{x}, \tilde{y}) + \ca{N}_{\X}(\tilde{x}),\\
            & 0 \in \partial_y h(\tilde{x}, \tilde{y}) + \ca{N}_{\Y}(\tilde{y}). 
        \end{aligned}
        \right. 
    \end{equation}
    Moreover, we say that $(\tilde{x}, \tilde{y}) \in \X \times \Y$ is an $\varepsilon$-first-order stationary point of the minimization problem $\min_{x \in \X, ~y \in \Y}~ h(x, y)$ if 
    \begin{equation}
        \left\{
        \begin{aligned}
            & \mathrm{dist}\left(0, \partial_x h(\tilde{x}, \tilde{y}) + \ca{N}_{\X}(\tilde{x}) \right) \leq \varepsilon,\\
            & \mathrm{dist}\left(0, \partial_y h(\tilde{x}, \tilde{y}) + \ca{N}_{\Y}(\tilde{y}) \right) \leq \varepsilon. 
        \end{aligned}
        \right. 
    \end{equation}
\end{defin}

\section{Minimax Optimization with Coupled Constraints}
In this section, we show that under Assumption \ref{Assumption_f}, \eqref{Prob_Con_Minmax} and \eqref{Prob_Pen_Minmax} are equivalent in the sense of first-order, local, and global minimax points. Section 3.1 introduces several preliminary lemmas. Then in Section 3.2, we present the main results on the equivalence between \eqref{Prob_Con_Minmax} and \eqref{Prob_Pen_Minmax}. 

% Furthermore, Section 3.3 illustrates that the nondegeneracy condition (i.e., Assumption \ref{Assumption_f}(4b)) can be guaranteed almost surely, by random perturbation to the constraint $c(x, y) \in \K$ under mild conditions. 

\subsection{Preliminary lemmas}
In this subsection, we present several preliminary lemmas for our theoretical analysis. We begin with the following two lemmas on the gradient inequality of the convex functions, which appear in \cite[Theorem 12.59]{rockafellar2009variational}. 
\begin{lem}
    \label{Le_subgradient_inequality}
    For any $\mu$-strongly convex function $r: \Rn \to \bb{R}$, and closed convex subset $\Omega$ of $\Rn$, it holds that for any $x_1, x_2 \in \Omega$,  
    \begin{equation}
        r(x_1) \geq r(x_2) + \inner{d, x_1-x_2} + \frac{\mu}{2} \norm{x_1 - x_2}^2,\quad \forall d \in \partial r(x_2) + \ca{N}_{\Omega}(x_2). 
    \end{equation}
\end{lem}

% \begin{lem}
%     \label{Le_subgradient_montone}
%     For any $\mu$-strongly convex function $r: \Rn \to \bb{R}$, and closed convex subset $\Y$ of $\Rp$, any $y_1, y_2 \in \Y$, $d_1 \in \partial r(y_1) + \ca{N}_{\Y}(y_1)$, and $d_2 \in \partial r(y_2) + \ca{N}_{\Y}(y_2)$,  it holds that 
%     \begin{equation}
%         \inner{y_1 - y_2, d_1 - d_2} \geq \mu \norm{y_1 - y_2}^2.
%     \end{equation}
% \end{lem}

The following lemma illustrates the differentiablity of $\mathrm{dist}(x, \K)^2$, which directly follows from \cite[Section 3.3, Exercise 12(d)]{borwein2006convex}. 
\begin{lem}
    \label{Le_diff_dist}
    For any closed convex cone $\K$ in $\Rn$, let the function $p_{\K}(x)$ be defined as $p_{\K}(x) = \frac{1}{2} \mathrm{dist}\left( x, \K \right)^2$, then it holds that 
    \begin{equation}
        \nabla p_{\K}(x) = \Pi_{\K^{\circ}}(x). 
    \end{equation}
\end{lem}

For any $(x, y) \in \X \times \Y$ satisfying $c(x, y) \in \K$, let  $\tilde{\tau}_{(x,y)}$ be defined as 
\begin{equation}
    \tilde{\tau}_{(x, y)} :=  \inf_{ \{(\lambda,\nu)\, :\, \lambda \in \ca{N}_{\K}(c(x, y)),\, \nu \perp \mathrm{lin}(\ca{T}_{\Y}(y)), \, \norm{\lambda}^2 + \norm{\nu}^2 =1\} }
        \norm{\nabla_y c(x, y) \lambda + \nu},
\end{equation}
then the following lemma presents the nondegeneracy of the Jacobian of the mapping $y \mapsto c(x, y)$. 
\begin{lem}
    \label{Le_nondegeneracy}
    Suppose Assumption \ref{Assumption_f} holds, then for any $(x, y) \in \X \times \Y$ satisfying $c(x, y) \in \K$, it holds that $\tilde{\tau}_{(x, y)} > 0$. 
    Moreover, for any compact subsets $\tilde{\X} \subseteq \X$, $\tilde{\Y} \subseteq \Y$, it holds that 
    \begin{equation}
        \inf_{x \in \tilde{\X}, y\in \tilde{\Y}, c(x, y)\in \K} \tilde{\tau}_{(x, y)} > 0. 
    \end{equation}
\end{lem}
\begin{proof}
    For any $(x, y) \in \X \times \Y$ that satisfies $c(x, y) \in \K$, we denote 
    \begin{equation}
        {\tau}_{(x, y)} := \max \left\{\tau\geq 0: \ca{B}_{\tau}  \subseteq  
        \left[
        \begin{smallmatrix}
            \nabla_y c(x, y)\tp\\
            I_p\\
        \end{smallmatrix}
        \right] \ca{B}_1 
        + 
        \mathrm{lin}\left(
        \left[
        \begin{smallmatrix}
            \ca{T}_{\K}(c(x, y))\\
            \ca{T}_{\Y}(y)\\
        \end{smallmatrix}
        \right]
        \right) \right\},
    \end{equation}
    where ${\cal B}_\tau\subset \mathbb{R}^{m+p}$ and ${\cal B}_1\subset\mathbb{R}^p$ denote the closed balls centered at the origin with radius $\tau$ and $1$, respectively. 

     We claim that under Assumption \ref{Assumption_f}(4), it holds that 
     ${\tau}_{(x, y)} > 0$, which we prove by contradiction next. 
     Suppose that ${\tau}_{(x, y)} = 0$. Denote $\ca{Z} = \left[
        \begin{smallmatrix}
            \nabla_y c(x, y)\tp\\
            I_p\\
        \end{smallmatrix}
        \right] \ca{B}_1 
        + 
        \mathrm{lin}\left(
        \left[
        \begin{smallmatrix}
            \ca{T}_{\K}(c(x, y))\\
            \ca{T}_{\Y}(y)\\
        \end{smallmatrix}
        \right]
        \right)$. 
        Then ${\tau}_{(x, y)} = 0$ implies that there exists $\tau >0$ and $(\tilde{x},\tilde{y}) \in B_\tau \setminus \ca{Z}.$
        By the Hahn-Banach theorem, there exist $w \in \bb{R}^{m+p} \setminus \{0\}$ such that $\inner{w, \ca{Z}} \leq 0$.    
        Since $\ca{Z} = - \ca{Z}$, so we can conclude that $\inner{w, \ca{Z}}=\inner{w, -\ca{Z}} =\{0\}$. 
    Then it holds that $\inner{w, \left[
        \begin{smallmatrix}
            \nabla_y c(x, y)\tp\\
            I_p\\
        \end{smallmatrix}
        \right] \Rp +  \mathrm{lin}\left(
        \left[
        \begin{smallmatrix}
            \ca{T}_{\K}(c(x, y))\\
            \ca{T}_{\Y}(y)\\
        \end{smallmatrix}
        \right]
        \right)} = \{0\}$, 
      %  hence $\inner{w, \left[
      %  \begin{smallmatrix}
      %      \nabla_y c(x, y)\tp\\
      %      I_p\\
      %  \end{smallmatrix}
      %  \right] \Rp } = \{0\}$,  
        which contradicts Assumption \ref{Assumption_f}(4). Thus we can conclude that ${\tau}_{(x, y)} > 0$. 
        
    % Combined with Assumption \ref{Assumption_f}(4), it holds that $\tilde{\tau}_{(x, y)} > 0$. Otherwise, we claim that $\tilde{\tau}_{(x, y)} = 0$, then the Hahn-Banach theorem illustrates that there exist $w \in \bb{R}^{m+p} \setminus \{0\}$ such that $\inner{w, \ca{Z}} \leq 0$.    
    %     Here we denote $\ca{Z} = \left[
    %     \begin{smallmatrix}
    %         \nabla_y c(x, y)\tp\\
    %         I_p\\
    %     \end{smallmatrix}
    %     \right] \ca{B}_1 
    %     + 
    %     \mathrm{lin}\left(
    %     \left[
    %     \begin{smallmatrix}
    %         \ca{T}_{\K}(c(x, y))\\
    %         \ca{T}_{\Y}(y)\\
    %     \end{smallmatrix}
    %     \right]
    %     \right)$ in the proof of this lemma. 
    %     Notice that for any $\tilde{w} \in \ca{Z}$, $-\tilde{w}$  also lies in $\ca{Z}$, this fact illustrates that $
    %     \inner{w, \ca{Z}} =\{0\}$. 
    % Then it holds that $\inner{w, \left[
    %     \begin{smallmatrix}
    %         \nabla_y c(x, y)\tp\\
    %         I_p\\
    %     \end{smallmatrix}
    %     \right] \ca{B}_1 } = \{0\}$, hence $\inner{w, \left[
    %     \begin{smallmatrix}
    %         \nabla_y c(x, y)\tp\\
    %         I_p\\
    %     \end{smallmatrix}
    %     \right] \Rn } = \{0\}$. Therefore, we have $
    %     \inner{w, \ca{Z}} = \{0\}$,  which contradicts Assumption \ref{Assumption_f}(4). Then we can conclude that our claim is invalid, and thus $\tilde{\tau}_{(x, y)} > 0$ holds for any $(x, y) \in \X \times \Y$. 
    
    Now, observe that for any $\lambda \in \ca{N}_{\K}(c(x, y))$ and any $\nu \perp \mathrm{lin}(\ca{T}_{\Y}(y))$ that satisfy $\norm{\lambda}^2 + \norm{\nu}^2 = 1$, it holds that  $(\lambda, \nu) \perp \mathrm{lin}\left(\left[
        \begin{smallmatrix}
            \ca{T}_{\K}(c(x, y))\\
            \ca{T}_{\Y}(y)\\
        \end{smallmatrix}
        \right] \right)$. 
    Then we have that 
    \begin{equation*}
        {\tau}_{(x, y)} \cdot(\lambda , \nu)
        \in
        \left[
        \begin{smallmatrix}
            \nabla_y c(x, y)\tp\\
            I_p\\
        \end{smallmatrix}
        \right] \ca{B}_1 .
    \end{equation*}
    Therefore, it holds that 
    \begin{equation*}
        \begin{aligned}
            &{\tau}_{(x, y)}  = \inner{(\lambda , \nu), \tilde{\tau}_{(x, y)} (\lambda , \nu)} \leq \sup_{\norm{u} = 1}\inner{u, \nabla_y c(x, y) \lambda + \nu} = \norm{\nabla_y c(x, y) \lambda + \nu}. 
        \end{aligned}
    \end{equation*}
    Thus we get $\tilde{\tau}_{(x,y)} \geq \tau_{(x,y)} > 0$.
    This completes the first part of the proof.

    Next we aim to prove the second part of this lemma by contradiction. For any compact subsets $\tilde{\X} \subseteq \X$, $\tilde{\Y} \subseteq \Y$, suppose there exists $\{(x_k, y_k)\} \subseteq \tilde{\X} \times \tilde{\Y}$ such that $c(x_k, y_k) \in \K$ and $\tilde{\tau}_{(x_k, y_k)} \to 0$. Then for any $k\geq 0$, there exists $\lambda_k \in \ca{N}_{\K}(c(x_k, y_k))$ and $\nu_k \perp \mathrm{lin}(\ca{T}_{\Y}(\yk))$ that satisfy $\norm{\lambda_k}^2 + \norm{\nu_k}^2 = 1$ and 
    \begin{equation}
        \norm{\nabla_y c(x_k, y_k) \lambda_k + \nu_k} < 2 \tilde{\tau}_{(x_k, y_k)}.
    \end{equation}
    Then from the graph-closedness of $\ca{N}_{\K}$ and $\ca{N}_{\Y}$, there exists $(x^*, y^*)$, $\lambda^* \in \ca{N}_{\K}(c(x^*, y^*))$, $\nu^* \perp \mathrm{lin}(\ca{T}_{\Y}(y^*))$, and a subsequence $\{i_k\}$ such that $(x_{i_k}, y_{i_k}) \to (x^*, y^*)$ and $(\lambda_{i_k}, \nu_{i_k}) \to (\lambda^*, \nu^*)$. Therefore, we have 
    \begin{equation}
        \norm{\nabla_y c(x^*, y^*) \lambda^* + \nu^*} = 0.
    \end{equation}
    This contradicts Assumption \ref{Assumption_f}(4).  This completes the proof of the lemma. 
\end{proof}

Next, we present the following lemma illustrating that $\lambda^*(x)$ is a singleton for any $x \in \X$. Based on the results in Lemma \ref{Le_unique_lambda_star}, in the rest of this paper, we treat $\lambda^*(\cdot)$ as a mapping from $\Rn$ to $\bb{R}^m$. 
\begin{lem}
    \label{Le_unique_lambda_star}
    Suppose Assumption \ref{Assumption_f} holds, then for any $x \in \X$, $\lambda^*(x)$ is a singleton.
\end{lem}
\begin{proof}
    Suppose $\lambda^*(x)$ is not singleton, then for any $x \in \X$,  from the optimality condition of the inner maximization subproblem of \eqref{Prob_Con_Minmax}, for any $\lambda_1, \lambda_2 \in \lambda^*(x)$, there exists $\nu_1, \nu_2 \in \ca{N}_{\Y}(y^*(x))$ such that 
    \begin{equation}
        0 = -\nabla_y g(x, y^*(x)) + \nabla_y c(x, y^*(x)) \lambda_i + \nu_i, \quad \forall i \in \{1,2\}.
    \end{equation}
    Therefore, we have $0 = \nabla_y c(x, y^*(x))(\lambda_1 - \lambda_2) + (\nu_1 - \nu_2)$. Notice that $(\nu_1 - \nu_2) \perp \mathrm{lin}(\ca{T}_{\Y}(y^*(x)))$, it directly follows from Lemma \ref{Le_nondegeneracy} that $\norm{\lambda_1 -\lambda_2} = 0$. This completes the proof. 
\end{proof}

In the rest of this subsection, we introduce the concept of level-boundedness \cite[Definition 1.16]{rockafellar2009variational} whose definition is given below. 
\begin{defin}
    A function $f: \X \times \Y \to \bb{R}$ is level-bounded in $y$ and locally uniformly in $x$, if for any $x \in \X$ and any $\alpha \in \bb{R}$, there exists a neighborhood $\ca{V}$ of $x$ such that the set $\{(x, y) \in \X \times \Y: x \in \ca{V}, f(x, y) \leq \alpha\}$ is bounded in $\X \times \Y$. 
\end{defin}

In the following lemma, we show the level-boundedness of $-L_{P}(x, \lambda, y)$ defined in \eqref{eq:LHG}. The result in Lemma \ref{Le_levelbounded_locallyuniformly} directly follows from the strong convexity of $-L_{P}$ with respect to $y$. Thus we omit its proof for simplicity. 
\begin{lem}
    \label{Le_levelbounded_locallyuniformly}
    Suppose Assumption \ref{Assumption_f} holds, then the function $-L_{P}(x, \lambda, y)$ in \eqref{eq:LHG}
    is level-bounded in $y$ and locally uniformly in $(x, \lambda)$.
\end{lem}

\subsection{Equivalence between (MCC) and (MOL)}
In this subsection, we aim to show that \eqref{Prob_Con_Minmax} and \eqref{Prob_Pen_Minmax} have the same first-order, local, and global minimax points. Therefore, applying first-order descent-ascent methods to solve \eqref{Prob_Con_Minmax} through \eqref{Prob_Pen_Minmax} is guaranteed to find the first-order minimax points of \eqref{Prob_Con_Minmax}. 

We first present the following lemma illustrating the strong duality of the inner subproblem in \eqref{Eq_intro_reformulate}, which directly follows from  \cite[Theorem 37.3]{rockafellar1970convex} and \cite[Proposition 2.156 and page 105]{bonnans2013perturbation}. Hence the proof is omitted for simplicity. 
\begin{lem}
    \label{Le_strong_duality}
    Suppose Assumption \ref{Assumption_f} holds, then for any $x \in \X$ and any $\rho \in (0, \infty)$, it holds that 
    \begin{equation}
        \max_{y \in \Y} \min_{\lambda \in \K^{\circ} \cap \ca{B}_{\rho}}  ~L(x, \lambda, y) = \min_{\lambda \in \K^{\circ} \cap \ca{B}_{\rho}} \max_{y \in \Y} ~  L(x, \lambda, y), \quad \text{and} \quad G(x) = G_T(x). 
    \end{equation}
\end{lem}

%Moreover, we have the following auxiliary lemma illustrating the local boundedness of the mapping $y^*$. 
The following lemma is a direct corollary of the result in \cite[Theorem 4.51, Proposition 4.47]{bonnans2013perturbation}. Hence we omit the proof for simplicity. 
\begin{lem}
    Suppose Assumption \ref{Assumption_f} holds, then the mapping $y^*$ is continuous over $\X$. 
\end{lem}

With the notation of $\tilde{\tau}_{(x, y)}$ as in Lemma \ref{Le_nondegeneracy} for any given $(x, y) \in \X \times \Y$, the following proposition characterizes the local boundedness of the mapping $\lambda^*(x)$ defined in \eqref{eq:G}.
\begin{prop}
    \label{Prop_bound_tau}
    Suppose Assumption \ref{Assumption_f} holds. Then for any $x \in \X$, we have 
    \begin{equation}
        c(x, y^*(x)) \in \K, \quad \inner{ \lambda^*(x), c(x, y^*(x)) } = 0, \quad \norm{\lambda^*(x)} \leq \frac{\norm{\nabla_y g(x, y^*(x))}}{\tilde{\tau}_{(x, y^*(x))}}. 
    \end{equation}
\end{prop}
\begin{proof}
    For any $x \in \X$, from the optimality condition of the inner maximization subproblem in \eqref{Prob_Con_Minmax}, there exists $\nu^* \in \ca{N}_{\Y}(y^*(x))$ such that, 
    \begin{equation}
        \label{Eq_Prop_bound_tau_0}
        \begin{aligned}
            & c(x, y^*(x)) \in \K, \quad \inner{\lambda^*(x) , c(x, y^*(x)) } = 0,\\
            & - \nabla_y g(x, y^*(x))  + \nu^* + \nabla_y c(x, y^*(x)) \lambda^*(x)=0. 
        \end{aligned}
    \end{equation}
    Then Lemma \ref{Le_nondegeneracy} illustrates that
    \begin{equation}
        \begin{aligned}
            &\tilde{\tau}_{(x, y^*(x))}^2 \norm{\lambda^*(x)}^2 \leq \tilde{\tau}_{(x, y^*(x))}^2 (\norm{\lambda^*(x)}^2 + \norm{\nu^*}^2) \\
            \leq{}&  \norm{ \nu^* + \nabla_y c(x, y^*(x)) \lambda^*(x)}^2 = \norm{ \nabla_y g(x, y^*(x)) }^2. 
        \end{aligned}
    \end{equation}
    This completes the proof. 
\end{proof}

Moreover, for any $\rho > 0$, we consider the following optimization problem:
%\begin{equation}
%    \label{Eq_minimize_HT_restrict}
%    \min_{\lambda \in \K^{\circ} \cap \ca{B}_\rho} H_T(\tilde{x}, \lambda),
%\end{equation}
%and define
\begin{equation} \label{Eq_minimize_HT_restrict}
    G_{T, \rho}(x) := \min_{\lambda \in \K^{\circ} \cap \ca{B}_\rho} H_T(x, \lambda). 
\end{equation}
Then we have the following proposition illustrating the equivalence between $G(x)$ and $G_{T, \rho}(x)$ for sufficiently large $\rho$.
\begin{prop}
    \label{Prop_GT_local_equivalent}
    Suppose Assumption \ref{Assumption_f} holds. Then for any $x \in \X$, there exists $\rho_x > 0$ and $\delta_x > 0$ such that for all $\tilde{x} \in \X \cap \ca{B}_{\delta_x}(x)$, the optimization problem \eqref{Eq_minimize_HT_restrict} admits $\lambda^*(\tilde{x})$ as its global minimizer. Moreover,  $G_{T, \rho_x}(\tilde{x}) = G(\tilde{x})$ holds for all $\tilde{x} \in \X \cap \ca{B}_{\delta_x}(x)$. 
\end{prop}
\begin{proof}
    For any $x \in \X$, for the mapping 
    $\lambda^*(x)$ in \eqref{eq:G}, we define the an upper-bound for $\norm{\lambda^*(x)}$ as $U_{\lambda}(x) = \frac{2\norm{\nabla_y g(x, y^*(x)) }}{\tilde{\tau}_{(x, y^*(x))}}$ from Proposition \ref{Prop_bound_tau}. Then from Lemma \ref{Le_nondegeneracy}, we can conclude that the function $U_{\lambda}(x)$ is locally bounded over $\X$. Therefore, for any $x \in \X$ and any $\delta_x > 0$, let $\rho_x > \sup_{\tilde{x} \in \X \cap \ca{B}_{\delta_x}(x)} U_{\lambda(\tilde{x})}$, it holds that 
    \begin{equation}
        \sup_{\tilde{x} \in \X \cap \ca{B}_{\delta_x}(x)} \norm{\lambda^*(\tilde{x})} < \rho_x. 
    \end{equation}
    % Moreover, from the convexity of $\phi(x, \cdot)$, the function $\lambda \mapsto L(x, \lambda, y^*(x, \lambda))$ is convex over $\K^{\circ}$, and $\lambda^*(x)$ is one of its global minimizer. 

    Now, for the optimization problem \eqref{Eq_minimize_HT_restrict}, we claim that it admits $\lambda^*(\tilde{x})$ as its unique global minimizer. Otherwise, suppose there exists another global minimizer $\lambda_1 \neq \lambda^*(\tilde{x})$. From the fact that $\norm{\lambda^*(\tilde{x})} < \rho_x$, we can conclude that there exists $\lambda_2 \neq \lambda^*(\tilde{x})$, which is a convex combination of $\lambda_1$ and $\lambda^*(x)$, and satisfies  $\norm{\lambda_2} < \rho_x$. Therefore, from the optimality condition in \eqref{Eq_Prop_bound_tau_0}, we can conclude that $\lambda_2$ is a multiplier for the inner maximization subproblem of \eqref{Prob_Con_Minmax}. Then from the uniqueness of the multiplier in Lemma \ref{Le_unique_lambda_star}, we can conclude that $\lambda_2 = \lambda^*(\tilde{x})$, which leads to the contradiction. Therefore, we can conclude that \eqref{Eq_minimize_HT_restrict} admits $\lambda^*(\tilde{x})$ as its unique global minimizer. This completes the first part of the proof. 

    Furthermore, from the uniqueness of $\lambda^*(\tilde{x})$ for any fixed $\tilde{x}$, it holds that 
    \begin{equation}
        G_{T, \rho_x}(\tilde{x}) = \inf_{\lambda \in \K^{\circ} \cap \ca{B}_{\rho_x}} H_T(\tilde{x}, \lambda) = \inf_{\lambda \in \K^{\circ} } H_T(\tilde{x}, \lambda) = G(\tilde{x}). 
    \end{equation}
    This completes the proof. 
\end{proof}

Next, based on the results of differentiating value functions \cite{rockafellar2009variational,mordukhovich2009subgradients}, the following lemma illustrates the differentiability of $H_T$ and the formulation of its gradients. 
\begin{lem}
    \label{Le_gradient_HT}
    Suppose Assumption \ref{Assumption_f} holds. Then the function $H_T$ is differentiable and
    \begin{equation}
        \begin{aligned}
            \nabla_x H_T(x, \lambda) ={}& \nabla_x  g(x, y^*_T(x, \lambda)) - \nabla_x c(x, y^*_T(x, \lambda))\lambda,\\
            \nabla_{\lambda} H_T(x, \lambda) ={}& -c(x, y^*_T(x, \lambda)).
        \end{aligned}
    \end{equation}
\end{lem}
\begin{proof}
    As illustrated in Lemma \ref{Le_levelbounded_locallyuniformly},  $-L_{P}$ is level-bounded in $y$ and locally uniformly 
    in $(x,\lambda)$.
    Moreover, for any $(x, \lambda) \in \X \times \K^{\circ}$, it holds that 
    \begin{equation*}
        -H_T(x, \lambda) = \min_{y \in \Y} - L_{P}(x, \lambda, y). 
    \end{equation*}
    Then it directly follows from \cite[Theorem 10.58]{rockafellar2009variational} that the function $-H_T$ is differentiable at $(x, \lambda) \in \X \times \K^{\circ}$, and 
    \begin{equation}
       \begin{aligned}
            \nabla_x H_T(x, \lambda) ={}& \nabla_x L_{P}(x, \lambda, y^*_T(x, \lambda)) =  \nabla_x  g(x, y^*_T(x, \lambda)) - \nabla_x c(x, y^*_T(x, \lambda))\lambda,\\
            \nabla_{\lambda} H_T(x, \lambda) ={}&  \nabla_{\lambda} L_{P}(x, \lambda, y^*_T(x, \lambda)) = -c(x, y^*_T(x, \lambda)).
        \end{aligned}
    \end{equation}
    This completes the proof. 
\end{proof}

Moreover, Proposition \ref{Prop_GT_local_equivalent} also illustrates that for any $x \in \X$, $\lambda^*_T(x)$ is a singleton and $\lambda^*_T(x) = \lambda^*(x)$. Therefore, in the rest of this paper, we treat $\lambda^*_T$ as a mapping from $\Rn$ to $\Rp$. Then based on Proposition \ref{Prop_GT_local_equivalent} and Lemma \ref{Le_gradient_HT}, we have the following proposition on the differentiability of $G(x)$. 
\begin{prop}
    \label{Prop_gradient_G}
    Suppose Assumption \ref{Assumption_f} holds. Then it holds that 
    \begin{equation}
        \nabla G(x) = \nabla_x H_T(x, \lambda_T^*(x)). 
    \end{equation}
\end{prop}
\begin{proof}
    For any $x \in \X$ and any $\delta_x > 0$, Proposition \ref{Prop_GT_local_equivalent} illustrates that there exists $\rho_x > 0$ such that $G_T(\tilde{x}) = G_{T, \rho_x}(\tilde{x})$ for any $\tilde{x} \in \X \cap \ca{B}_{\rho_x}(x)$. 
    
    Then from the compactness of $\X  \cap \ca{B}_{\rho_x}$ and \cite[Theorem 10.58]{rockafellar2009variational}, we can conclude that $G_{T, \rho_x}$ is differentiable for any $\tilde{x} \in \X \cap \ca{B}_{\rho_x}(x)$, and 
    \begin{equation}
        \label{Eq_Prop_gradient_G_0}
        \nabla G_{T, \rho_x}(\tilde{x}) =  \nabla_x H_T(\tilde{x}, \lambda^*_T(\tilde{x})).
    \end{equation}
    Moreover, from Lemma \ref{Le_strong_duality}, we have that $G(\tilde{x}) = G_{T, \rho_x}(\tilde{x}) = G_T(\tilde{x})$ holds for any $\tilde{x} \in \X \cap \ca{B}_{\rho_x}(x)$. Together with \eqref{Eq_Prop_gradient_G_0},  we can conclude that 
    \begin{equation}
        \nabla G(x) = \nabla_x H_T(x, \lambda^*_T(x)). 
    \end{equation}
    This completes the proof.

\end{proof}

The following proposition illustrates the equivalence between \eqref{Prob_Con_Minmax} and \eqref{Prob_Pen_Minmax} with respect to both local minimax points and global minimax points. The results in proposition \ref{Prop_Equivalence_MCC_local} directly follows from the fact that $\Phi(x) = \inf_{\lambda \in\K^{\circ}} H_T(x, \lambda) + r_1(x)$, hence is omitted for simplicity. 
\begin{prop}
    \label{Prop_Equivalence_MCC_local}
    Suppose Assumption \ref{Assumption_f} holds. Then for any $(\tilde{x},\tilde{y}) \in \X \times \Y$ that is a global minimax point of \eqref{Prob_Con_Minmax}, there exists $\tilde{\lambda} \in \K^{\circ}$ such that $(\tilde{x}, \tilde{\lambda}, \tilde{y})$ is a global minimax point of \eqref{Prob_Pen_Minmax}. Moreover, for any $(\tilde{x},\tilde{y}) \in \X \times \Y$ that is a local minimax point of \eqref{Prob_Con_Minmax}, there exists $\tilde{\lambda} \in \K^{\circ}$ such that $(\tilde{x}, \tilde{\lambda}, \tilde{y})$ is a local minimax point of \eqref{Prob_Pen_Minmax}.
\end{prop}

Finally, we present the following theorem illustrating that \eqref{Prob_Con_Minmax} and \eqref{Prob_Pen_Minmax} have the same first-order minimax points. 
\begin{theo}
    \label{The_Equivalence_MCC_FOSP}
    Suppose Assumption \ref{Assumption_f} holds. Then for any $(\tilde{x},\tilde{y}) \in \X \times \Y$ that is a first-order minimax point of \eqref{Prob_Con_Minmax}, there exists $\tilde{\lambda} \in \K^{\circ}$ such that $(\tilde{x}, \tilde{\lambda}, \tilde{y})$ is a first-order minimax point of \eqref{Prob_Pen_Minmax}. 
    
    Moreover, for any $(\tilde{x}, \tilde{\lambda}, \tilde{y}) \in \X \times \K^{\circ} \times \Y$ that is a first-order minimax point of \eqref{Prob_Pen_Minmax}, it holds that $(\tilde{x}, \tilde{y})$ is a first-order minimax point of \eqref{Prob_Con_Minmax}. 
\end{theo}
\begin{proof}
    For any $(\tilde{x},\tilde{y}) \in \X \times \Y$ that is a first-order minimax point of \eqref{Prob_Con_Minmax}, we have
    \begin{equation}
        0\in \partial \Phi(\tilde{x}) +  \ca{N}_{\X}(\tilde{x}) = \nabla G(\tilde{x}) + \partial r_1(\tilde{x}) + \ca{N}_{\X}(\tilde{x}), \quad \tilde{y} = y^*(\tilde{x}). 
    \end{equation}
    Firstly, from the definition of $y^*_T$ and $\lambda^*_T$, it holds that $y^*(x) = y^*_T(x, \lambda^*_T(x))$ for any $x \in \X$. Therefore, we can conclude that $\tilde{y} = y^*_T(\tilde{x}, \lambda^*_T(\tilde{x}))$. 
    
    Then from Proposition \ref{Prop_gradient_G}, and Lemma \ref{Le_gradient_HT}, we have
    \begin{equation}
        \begin{aligned}
            &0 \in \partial \Phi(\tilde{x}) +  \ca{N}_{\X}(\tilde{x})= \nabla G(\tilde{x}) + \partial r_1(\tilde{x}) + \ca{N}_{\X}(\tilde{x}) =  \nabla_x H_T(\tilde{x}, \lambda^*_T(\tilde{x})) + \partial r_1(\tilde{x}) + \ca{N}_{\X}(\tilde{x})\\
            ={}& \nabla_x  g(\tilde{x}, y^*_T(\tilde{x}, \lambda^*_T(\tilde{x}))) - \nabla_x c(\tilde{x}, y^*_T(\tilde{x} , \lambda^*_T(\tilde{x})))\lambda^*_T(\tilde{x}) + \partial r_1(\tilde{x}) + \ca{N}_{\X}(\tilde{x})\\
            ={}& \nabla_x  g(\tilde{x}, \tilde{y}) - \nabla_x c(\tilde{x}, \tilde{y})\lambda^*_T(\tilde{x}) + \partial r_1(\tilde{x}) + \ca{N}_{\X}(\tilde{x}). 
        \end{aligned}
    \end{equation}
    Moreover, from the optimality condition for the inner maximization subproblem in \eqref{Prob_Con_Minmax} and the definition of $\lambda^*_T$, it holds that 
    \begin{equation}
        0 \in -\nabla_y g(\tilde{x}, \tilde{y}) + \nabla_y c(\tilde{x}, \tilde{y})\lambda^*_T(\tilde{x}) + \ca{N}_{\Y}(\tilde{y}). 
    \end{equation}
    As a result, there exists $\tilde{\lambda} =\lambda^*_T(\tilde{x}) \in \K^{\circ}$ such that 
    \begin{equation}
        \left\{
        \begin{aligned}
            &0 \in \nabla_x  g(\tilde{x}, \tilde{y}) - \nabla_x c(\tilde{x}, \tilde{y})\tilde{\lambda} + \partial r_1(\tilde{x}) + \ca{N}_{\X}(\tilde{x}),\\
            &0 \in -\nabla_y g(\tilde{x}, \tilde{y}) + \nabla_y c(\tilde{x}, \tilde{y})\tilde{\lambda} + \ca{N}_{\Y}(\tilde{y}),\\
            &0 \in -c(\tilde{x}, \tilde{y}) + \ca{N}_{\K^{\circ}}(\tilde{\lambda}),
        \end{aligned}
        \right.
    \end{equation}
    where the third condition holds because $c(\tilde{x},\tilde{y})\in \K$.
    Therefore, we can conclude that $(\tilde{x}, \tilde{\lambda}, \tilde{y})$ is a first-order minimax point of \eqref{Prob_Pen_Minmax}. 

    On the other hand, for any $(\tilde{x}, \tilde{\lambda}, \tilde{y}) \in \X \times \K^{\circ} \times \Y$ that is a first-order minimax point of \eqref{Prob_Pen_Minmax}, it holds that 
    \begin{equation}
        \left\{
        \begin{aligned}
            &0 \in \nabla_x  g(\tilde{x}, \tilde{y}) - \nabla_x c(\tilde{x}, \tilde{y})\tilde{\lambda} + \partial r_1(\tilde{x}) + \ca{N}_{\X}(\tilde{x}),\\
            &0 \in -\nabla_y g(\tilde{x}, \tilde{y}) + \nabla_y c(\tilde{x}, \tilde{y})\tilde{\lambda} + \ca{N}_{\Y}(\tilde{y}),\\
            &0 \in -c(\tilde{x}, \tilde{y}) + \ca{N}_{\K^{\circ}}(\tilde{\lambda}).
        \end{aligned}
        \right.
    \end{equation}
    Then from the second condition above, we can conclude that $\tilde{y} = y^*(\tilde{x})=y_T^*(\tilde{x},\lambda_T^*(\tilde{x}))$ and $\tilde{\lambda} = \lambda^*_T(\tilde{x})$. Moreover, from Proposition \ref{Prop_gradient_G}, and Lemma \ref{Le_gradient_HT}, it holds that 
    \begin{equation}
        \begin{aligned}
            &\partial \Phi(\tilde{x}) +  \ca{N}_{\X}(\tilde{x})= \nabla G(\tilde{x}) + \partial r_1(\tilde{x}) + \ca{N}_{\X}(\tilde{x}) =  \nabla_x H_T(\tilde{x}, \lambda^*_T(\tilde{x})) + \partial r_1(\tilde{x}) + \ca{N}_{\X}(\tilde{x})\\
            ={}& \nabla_x  g(\tilde{x}, y^*_T(\tilde{x}, \lambda^*_T(\tilde{x}))) - \nabla_x c(\tilde{x}, y^*_T(\tilde{x}, \lambda^*_T(\tilde{x})))\lambda^*_T(\tilde{x}) + \partial r_1(\tilde{x}) + \ca{N}_{\X}(\tilde{x})\\
            ={}&  \nabla_x  g(\tilde{x}, \tilde{y}) - \nabla_x c(\tilde{x}, \tilde{y})\lambda^*_T(\tilde{x}) + \partial r_1(\tilde{x}) + \ca{N}_{\X}(\tilde{x})\ni{}0. 
        \end{aligned}
    \end{equation}
    Therefore, $(\tilde{x}, \tilde{y})$ is a first-order minimax point of \eqref{Prob_Con_Minmax}. This completes the proof. 
\end{proof}

\section{Minimization Formulation of (MM) through (PFBE)}
In this section, we present the equivalence between the minimax problem \eqref{Prob_Minmax} and the minimization problem \eqref{Prob_Pen}. As the minimax problem \eqref{Prob_Minmax} generalizes the minimax 
problem \eqref{Prob_Pen_Minmax}, all the theoretical results in this section can be directly applied to \eqref{Prob_Pen_Minmax}, hence contributing to the 
computation of the minimax problem with coupled constraints \eqref{Prob_Con_Minmax}.

For the partial forward-backward envelope (PFBE), we define 
\begin{equation}\label{eq:PGdef}
    \left\{
    \begin{aligned}
        % &T_{x, \eta}(x, y) := \argmin_{u\in \X} f(x, y) + \inner{\nabla_x f(x, y), u-x}  + \frac{1}{2\eta} \norm{u-x}^2,\\
        & T_{y, \eta}(x, y) := \argmax_{ v\in \Y} ~f(x, y) + \inner{\nabla_y f(x, y), v-y}  - r_2(v) - \frac{1}{2\eta} \norm{v-y}^2,\\
        & R_{y, \eta}(x, y) := \frac{1}{\eta}\left(  T_{y, \eta}(x, y) - y \right).
        % & T^{M}_{y, \eta}(x, y) := \argmax_{ v\in \Y} f(x, v)  - \frac{1}{2\eta} \norm{v-y}^2, \quad R^{M}_{y, \eta}(x, y) := \frac{1}{\eta}(T^{M}_{y, \eta}(x, y) - y). \\
    \end{aligned}
    \right.
\end{equation}

Based on Assumption \ref{Assumption_f} and the assumptions from \cite{jin2020local,hassan2018non,lin2020gradient,lu2020hybrid,luo2022finding,xu2023unified}, we make the following standard assumptions on \eqref{Prob_Minmax}. 
\begin{assumpt}
    \label{Assumption_f_Minmax}
    \begin{enumerate}
        \item $f$ is continuously differentiable over $\Rn \times \Rp$, and  $\nabla f$ is $L_f$-Lipschitz continuous over $\Rn \times \Rp$. Moreover, $\nabla^2_{yy} f(x, y)$ exists for any $(x, y) \in \X \times \Y$.  
        \item For any $x \in \X$, the function $y \mapsto  f(x, y)$ is $\mu$-strongly concave, where $\mu>0$ is independent of $x.$ 
        \item The functions $r_1: \Rn \to \bb{R}$ is locally Lipschitz continuous, and $r_2: \Rp \to \bb{R}$ is locally Lipschitz continuous and convex. 
        \item For any $(x, y) \in \X \times \Y$, the proximal subproblem in \eqref{eq:PGdef} can be efficiently computed. 
    \end{enumerate}
\end{assumpt}
Assumption \ref{Assumption_f_Minmax}(1)-(3) are standard assumptions in various existing works \cite{jin2020local,hassan2018non,lin2020gradient,lu2020hybrid,luo2022finding,xu2023unified,cohen2024alternating}. Moreover, Assumption \ref{Assumption_f_Minmax}(4), which also appears in \cite{cohen2024alternating}, assumes  the regularization term $r_2$ is prox-friendly. It is worth mentioning that based on the results in \cite{yu2013decomposing}, the proximal subproblem in \eqref{eq:PGdef} can be efficiently computed for a wide range of regularization terms $r_2$ and constraint sets $\Y$.

\subsection{Basic properties}
In this subsection, we present the basic properties of the PFBE of the minimax problem \eqref{Prob_Minmax}. We first present some fundamental properties of the PFBE in the following lemma. 
\begin{lem}
    Suppose Assumption \ref{Assumption_f_Minmax} holds. Then given any $(x, y) \in \X \times \Y$ and $\eta>0$, the following inequalities hold, 
    \begin{itemize}
        \item $\Psi_{\eta}(x, y) \geq f(x, y) - r_2(y) + \frac{\eta}{2} \norm{R_{y, \eta}(x, y)}^2$;
        \item $f(x, T_{y, \eta}(x, y)) - r_2(T_{y, \eta}(x, y)) \geq \Psi_{\eta}(x, y) + \frac{\eta(1-\eta L_f)}{2} \norm{R_{y, \eta}(x, y)}^2$. 
    \end{itemize}
\end{lem}
\begin{proof}
    From the definition of $\Psi_{\eta}$ in \eqref{Eq_FBE_Y}, for any $(x, y) \in \X \times \Y$, it holds that,  
    \begin{equation}
        \begin{aligned}
            0 \in{}& \frac{1}{\eta}(T_{y, \eta}(x, y) - y) - \nabla_y f(x, y) + \partial r_2 (T_{y, \eta}(x, y)) + \ca{N}_{\Y}(T_{y, \eta}(x, y))\\
            ={}& R_{y, \eta}(x) - \nabla_y f(x, y) + \partial r_2 (T_{y, \eta}(x, y)) + \ca{N}_{\Y}(T_{y, \eta}(x, y)).
        \end{aligned}
    \end{equation}
    Then from Lemma \ref{Le_subgradient_inequality} and the convexity of $r_2$ and $\Y$, we have the following inequality,
    \begin{equation}
        r_2(y) \geq r_2(T_{y, \eta}(x, y)) + \inner{d, y - T_{y, \eta}(x, y)}, \quad \forall d \in \partial r_2(T_{y, \eta}(x, y)) + \ca{N}_{\Y}(T_{y, \eta}(x, y)).  
    \end{equation}
    As a result, it holds that 
    \begin{equation}
        \begin{aligned}
            &r_2(y) \geq r_2(T_{y, \eta}(x, y)) - \inner{  \nabla_y f(x, y) - R_{y, \eta}(x),  T_{y, \eta}(x, y) - y} \\
            ={}& r_2(T_{y, \eta}(x, y)) - \inner{\nabla_y f(x, y), T_{y, \eta}(x, y) - y} + \frac{1}{\eta} \norm{T_{y, \eta}(x, y) - y}^2, 
        \end{aligned}
    \end{equation}
    which leads to the following inequality,  
    \begin{equation}
        \begin{aligned}
            &f(x, y) - r_2(y) \leq  f(x, y) - r_2(T_{y, \eta}(x, y)) + \inner{\nabla_y f(x, y), T_{y, \eta}(x, y) - y} - \frac{1}{\eta} \norm{T_{y, \eta}(x, y) - y}^2\\
            ={}& \Psi_{\eta}(x, y) - \frac{1}{2\eta} \norm{T_{y, \eta}(x, y) - y}^2 =  \Psi_{\eta}(x, y) - \frac{\eta}{2} \norm{R_{y, \eta}(x, y)}^2. 
        \end{aligned}
    \end{equation}
    This completes the first part of the proof. 

    Furthermore, from the $L_f$-Lipschitz continuity of $\nabla_y f$ in Assumption \ref{Assumption_f_Minmax}(1), it holds that 
    \begin{equation}
        \begin{aligned}
            &f(x, T_{y, \eta}(x, y)) - r_2(T_{y, \eta}(x, y)) \\
            \geq{}& f(x, y) - r_2(T_{y, \eta}(x, y)) + \inner{\nabla_y f(x, y), T_{y, \eta}(x, y) - y} - \frac{L_f}{2} \norm{T_{y, \eta}(x, y) - y}^2\\
            ={}& f(x, y) - r_2(T_{y, \eta}(x, y)) + \inner{\nabla_y f(x, y), T_{y, \eta}(x, y) - y} \\
            & - \frac{1}{2\eta}\norm{T_{y, \eta}(x, y) - y}^2 + \frac{1 - \eta L_f}{2\eta}\norm{T_{y, \eta}(x, y) - y}^2\\
            ={}& \Psi_{\eta}(x, y) +  \frac{\eta(1 - \eta L_f)}{2}\norm{R_{y, \eta}(x, y)}^2. 
        \end{aligned}
    \end{equation}
    This completes the proof. 
\end{proof}

The following lemma directly follows from the strong concavity of $f$ with respect to $y$ in Assumption \ref{Assumption_f_Minmax}(2). As a result, we omit its proof for simplicity. 
\begin{lem}
    Suppose Assumption \ref{Assumption_f_Minmax} holds. Then for any $y \in \Y$ and any $d \in \ca{T}_{\Y}(y)$, it holds that 
    \begin{equation}
        \inner{d, \nabla_{yy}^2 f(x, y) d} \leq -\mu \norm{d}^2. 
    \end{equation}
\end{lem}

In the following, we define the auxiliary function $\Xi_{(\eta, \alpha)}: \X \times \Y \to \bb{R}$ as 
\begin{equation}
    \label{Eq_defin_Xi}
    \Xi_{(\eta, \alpha)}(x, y) := \alpha \Psi_{\eta}(x, y) - (\alpha-1) f(x, y),
\end{equation}
which corresponds to the differentiable part of $\Gamma_{(\eta, \alpha)}$. 
Then we have the following proposition illustrating the gradients of $\Psi_{\eta}$ and $\Xi_{(\eta, \alpha)}$. As Proposition \ref{Prop_fundmental_differential} directly follows from \cite[Theorem 10.58]{rockafellar2009variational}, we omit its proof for simplicity. 
\begin{prop}
    \label{Prop_fundmental_differential}
    For any $(x, y) \in \Rn \times \Y$ and $\eta > 0$, it holds that 
    \begin{equation}
        \left\{
        \begin{aligned}
            \nabla_x \Psi_{\eta}(x, y)={}& \nabla_x f(x,y) + \eta \nabla_{xy}^2 f(x, y)R_{y, \eta}(x, y), \\
            \nabla_{y} \Psi_{\eta}(x, y)={}& \left( I_p + \eta \nabla_{yy}^2 f(x, y)  \right)R_{y, \eta}(x, y).
        \end{aligned}
        \right.
    \end{equation}
    Moreover, we have 
    \begin{equation}
        \left\{
        \begin{aligned}
            \nabla_x \Xi_{(\eta, \alpha)}(x, y)={}& \nabla_x f(x,y) + \alpha \eta \nabla_{xy}^2 f(x, y)R_{y, \eta}(x, y), \\
            \nabla_y \Xi_{(\eta, \alpha)}(x, y)={}& \left( I_p + \alpha \eta \nabla_{yy}^2 f(x, y)  \right)R_{y, \eta}(x, y) + (1-\alpha) (\nabla_y f(x, y) - R_{y, \eta}(x, y)).
        \end{aligned}
        \right.
    \end{equation}
\end{prop}

\subsection{Equivalence between (MM) and  (MMPen)}
In this subsection, we prove that with sufficiently large $\alpha \geq 1$, the minimization problem \eqref{Prob_Pen} is equivalent to the minimax problem \eqref{Prob_Minmax}. We begin our analysis with the following preliminary lemma. 
\begin{lem}
    \label{Le_gradient_inequality_FBEY}
    Suppose Assumption \ref{Assumption_f_Minmax} holds, and $\eta> 0$. Then for any  $(x, y) \in \X \times \Y$, it holds that 
    \begin{equation}
        \inner{R_{y, \eta}(x, y), R_{y, \eta}(x, y) -\nabla_y f(x, y) + \partial r_2(y) + \ca{N}_{\Y}(y)} \leq 0. 
    \end{equation}
\end{lem}
\begin{proof}
    It follows from the formulation of PFE with respect to the $y$-variable in \eqref{Eq_FBE_Y} that 
    \begin{equation}
        0 \in R_{y, \eta}(x, y) - \nabla_y f(x, y) + \partial r_2(T_{y, \eta}(x, y)) + \ca{N}_{\Y}(T_{y, \eta}(x, y)).
    \end{equation}
    Hence 
    \begin{equation}
         \nabla_y f(x, y) - R_{y, \eta}(x, y) \in \partial r_2(T_{y, \eta}(x, y)) + \ca{N}_{\Y}(T_{y, \eta}(x, y)).
    \end{equation}
    From the convexity of $r_2$, Lemma \ref{Le_subgradient_inequality} illustrates that 
    \begin{equation}
        \label{Eq_Le_gradient_inequality_FBEY_0}
        r_2(y) \geq r_2(T_{y, \eta}(x, y)) + \inner{\nabla_y f(x, y) - R_{y, \eta}(x, y), y - T_{y, \eta}(x, y)}.
    \end{equation}
    Then by reorganizing the terms in \eqref{Eq_Le_gradient_inequality_FBEY_0}, we have
    \begin{equation}
        \inner{R_{y, \eta}(x, y) - \nabla_y f(x, y),  T_{y, \eta}(x, y) - y} \leq r_2(y) - r_2(T_{y, \eta}(x, y)).
    \end{equation}
    Therefore, it holds that 
    \begin{equation}
        \begin{aligned}
            &\inner{T_{y, \eta}(x, y) - y, R_{y, \eta}(x, y) -\nabla_y f(x, y) + \partial r_2(y) + \ca{N}_{\Y}(y)}\\
            ={}&\inner{T_{y, \eta}(x, y) - y, R_{y, \eta}(x, y) - \nabla_y f(x, y)} +  \inner{T_{y, \eta}(x, y) - y,  \partial r_2(y) + \ca{N}_{\Y}(y)}\\
            \leq{}& r_2(y) - r_2(T_{y, \eta}(x, y)) + \inner{T_{y, \eta}(x, y) - y,  \partial r_2(y) + \ca{N}_{\Y}(y)} \\
            \leq{}& 0. 
        \end{aligned}
    \end{equation}
    Here the last inequality directly follows from Lemma \ref{Le_subgradient_inequality} with $x_1 = T_{y, \eta}(x, y)$ and $x_2 = y$.  
    This completes the proof. 
\end{proof}

Now, we are ready to present the following theorem illustrating that the first-order minimax points of \eqref{Prob_Minmax} coincide with the first-order stationary points of \eqref{Prob_Pen}. 
\begin{theo}
    \label{Theo_Equivalence_Gamma}
    Suppose Assumption \ref{Assumption_f_Minmax} holds,  $\eta> 0$, and $\alpha \geq \max\{1, \frac{2}{\eta \mu}\}$. Then for any $(x, y) \in \X \times \Y$, $(x, y)$ is a first-order minimax point of \eqref{Prob_Minmax} if and only if it is a first-order stationary point of \eqref{Prob_Pen}. 
\end{theo}
\begin{proof}
    We first prove the ``if'' part of the theorem. 
    From the optimality condition for \eqref{Prob_Pen} in Definition \ref{Def_FOSP_Pen}, we have
    \begin{equation}
        \label{Eq_Theo_Equivalence_Gamma_0}
        \left\{
        \begin{aligned}
            0 \in{}& \nabla_x f(x,y) + \alpha \eta \nabla_{xy}^2 f(x, y)R_{y, \eta}(x, y) + \partial r_1(x) + \ca{N}_{\X}(x), \\
            0 \in{}&  \left( I_p +  \alpha \eta \nabla_{yy}^2 f(x, y)  \right)R_{y, \eta}(x, y) + (\alpha-1) \left(R_{y, \eta}(x, y) -\nabla_y f(x, y) + \partial r_2(y)\right) + \ca{N}_{\Y}(y).
        \end{aligned}
        \right.
    \end{equation}

    Notice that $R_{y, \eta}(x, y) \in \ca{T}_{\Y}(y)$, we have
    \begin{equation}
        \label{Eq_Theo_Equivalence_Gamma_1}
        \inner{R_{y, \eta}(x, y), d} \leq 0, \quad \forall d \in \ca{N}_{\Y}(y). 
    \end{equation}
    Moreover,  from the optimality condition in \eqref{Eq_Theo_Equivalence_Gamma_0} with respect to the $y$-variable, there exists $d_n \in \ca{N}_{\Y}(y)$ and $d_y \in R_{y, \eta}(x, y) -\nabla_y f(x, y)  + \partial r_2(y) + \ca{N}_{\Y}(y)$ such that 
    \begin{equation}
        \label{Eq_Theo_Equivalence_Gamma_2}
        \begin{aligned}
            0={}&\inner{R_{y, \eta}(x, y),  \left( I_p +  \alpha\eta \nabla_{yy}^2 f(x, y)  \right)R_{y, \eta}(x, y) + (\alpha-1) d_y + d_n}\\
            \leq{}& \inner{R_{y, \eta}(x, y),  \left( I_p +  \alpha \eta \nabla_{yy}^2 f(x, y)  \right)R_{y, \eta}(x, y) } + (\alpha-1)\inner{R_{y, \eta}(x, y), d_y}\\
            \leq{}& \inner{R_{y, \eta}(x, y),  \left( I_p +  \alpha \eta \nabla_{yy}^2 f(x, y)  \right)R_{y, \eta}(x, y) }\\
            \leq{}&  - \frac{\mu \alpha \eta}{2} \norm{R_{y, \eta}(x, y)}^2 \leq 0. 
        \end{aligned}
    \end{equation}
    Here the first inequality uses \eqref{Eq_Theo_Equivalence_Gamma_1}, and the second inequality directly follows from Lemma \ref{Le_gradient_inequality_FBEY}. In addition, Assumption \ref{Assumption_f_Minmax} illustrates that $I_p + \alpha \eta \nabla_{yy}^2 f(x, y) \preceq (1-\alpha \eta\mu) I_p$, which leads to the last inequality of \eqref{Eq_Theo_Equivalence_Gamma_2}. 
    Thus we can conclude that $R_{y, \eta}(x, y) = 0$. Together with \eqref{Eq_Theo_Equivalence_Gamma_0}, we have 
    \begin{equation*}
        0 \in \nabla_x f(x, y) + \partial r_1(x) + \ca{N}_{\X}(x), \qquad 0 \in -\nabla_y f(x, y) + \partial r_2(y) +  \ca{N}_{\Y}(y).
    \end{equation*}
    This shows that $(x, y)$ is a first-order minimax point of \eqref{Prob_Minmax}.

    Next, we show the validity of the ``only if'' part of this theorem. For any $(x, y)$ that is a first-order minimax point of \eqref{Prob_Minmax}, it follows from Definition \ref{Def_FOSP_MinMax} that 
    \begin{equation*}
        0 \in \nabla_x f(x, y) + \partial r_1(x) + \ca{N}_{\X}(x), \qquad 0 \in -\nabla_y f(x, y) + \partial r_2(y) +  \ca{N}_{\Y}(y). 
    \end{equation*}
    This implies that $T_{y,\eta}(x,y)=y$, which leads to the fact that $y$ satisfies the optimality condition of (PFBE). Therefore, it holds that $R_{y, \eta}(x, y) = 0$. As a result, we have 
    \begin{eqnarray*}
           && \partial_x \Gamma_{(\eta, \alpha)} (x, y) = \nabla_x \Xi_{(\eta, \alpha)}(x, y) + \partial r_1(x) = \nabla_x f(x, y) + \partial r_1(x)  \in -  \ca{N}_{\X}(x),
\\[3pt]
        && \partial_y \Gamma_{(\eta, \alpha)} (x, y) = (\alpha - 1)\left( -\nabla_y f(x, y) + \partial r_2(y) \right) \in -  \ca{N}_{\Y}(y)  .
    \end{eqnarray*}
    Therefore, from Definition \ref{Def_FOSP_Pen}, we can conclude that $(x, y)$ is a first-order stationary point of \eqref{Prob_Pen}. This completes the entire proof of this theorem.  
\end{proof}

From the equivalence between \eqref{Prob_Minmax} and \eqref{Prob_Pen} in Theorem \ref{Theo_Equivalence_Gamma}, we can directly employ existing methods for minimization over $\X \times \Y$ to compute an $\varepsilon$-first-order stationary point of \eqref{Prob_Pen} with guaranteed complexity. Therefore, it is important to investigate the relationship between the $\varepsilon$-first-order stationary points of \eqref{Prob_Pen} and the $\varepsilon$-first-order minimax points of \eqref{Prob_Minmax}. 
The following theorem demonstrates that any $\varepsilon$-first-order stationary point of \eqref{Prob_Pen} is also a $(1 + \frac{2L_f}{\mu} + \eta L_f) \varepsilon$-first-order minimax point of \eqref{Prob_Minmax}.
\begin{theo}
    \label{Theo_eps_Equivalence_Gamma}
    Suppose Assumption \ref{Assumption_f_Minmax} holds,  $\eta> 0$, and $\alpha \geq \max\{1, \frac{2}{\eta \mu}\}$. Then for any $(x, y) \in \X \times \Y$ that is an $\varepsilon$-first-order stationary point of \eqref{Prob_Pen}, it holds that $(x, T_{y, \eta}(x, y))$ is a $(1+ \frac{2L_f}{\mu} + \eta L_f) \varepsilon$-first-order minimax point of \eqref{Prob_Minmax}.
\end{theo}
\begin{proof}
    For any $(x, y) \in \X \times \Y$ that is a $\varepsilon$-first-order stationary point of \eqref{Prob_Pen}, there exists $e_x \in \Rn$ and $e_y \in \Rp$ such that $\norm{e_x} \leq \varepsilon$, $\norm{e_y} \leq \varepsilon$ and  
    \begin{equation}
        \label{Eq_Theo_eps_Equivalence_Gamma_0}
        \left\{
        \begin{aligned}
            e_x \in{}& \nabla_x f(x,y) + \alpha \eta \nabla_{xy}^2 f(x, y)R_{y, \eta}(x, y) + \partial r_1(x) + \ca{N}_{\X}(x), \\
            e_y \in{}& \left( I_p + \alpha \eta \nabla_{yy}^2 f(x, y)  \right)R_{y, \eta}(x, y) + (\alpha-1 ) (R_{y, \eta}(x, y) - \nabla_y f(x, y) + \partial r_2(y)) + \ca{N}_{\Y}(y).
        \end{aligned}
        \right.
    \end{equation}
    Then similar to the techniques in \eqref{Eq_Theo_Equivalence_Gamma_2}, it holds that  
    \begin{equation*}
        \begin{aligned}
            &\varepsilon \norm{R_{y, \eta}(x, y)} \geq \inner{-R_{y, \eta}(x, y), e_y}\\
            \geq{}& -\inner{R_{y, \eta}(x, y),  \left( I_p + \alpha \eta \nabla_{yy}^2 f(x, y)  \right)R_{y, \eta}(x, y)} \\
            &+\inf_{d_{1} \in \partial r_2(y), ~d_{2} \in \ca{N}_{\Y}(y) }-\inner{R_{y, \eta}(x, y), (\alpha-1 ) (R_{y, \eta}(x, y) - \nabla_y f(x, y) + d_{1} ) + d_{2}}\\
            \geq{}& -\inner{R_{y, \eta}(x, y),  \left( I_p + \alpha \eta \nabla_{yy}^2 f(x, y)  \right)R_{y, \eta}(x, y)} \geq  \frac{\mu \alpha \eta}{2} \norm{R_{y, \eta}(x, y)}^2.
        \end{aligned}
    \end{equation*}
    Here the third inequality directly follows from Lemma \ref{Le_gradient_inequality_FBEY}. 
    Therefore, we can conclude that 
    \begin{equation*}
        \label{Eq_Theo_eps_Equivalence_Gamma_1}
        \norm{R_{y, \eta}(x, y)} \leq \frac{2}{\mu \alpha \eta} \varepsilon \leq \varepsilon. 
    \end{equation*}    
    Then from the $L_f$-Lipschitz continuity of $\nabla f$, we have the following estimation illustrating the stationarity with respect to the $y$-variable, 
    \begin{equation*}
        \begin{aligned}
            &\mathrm{dist}\left(0,  -\nabla_y f(x, T_{y, \eta}(x, y)) + \partial r_2(T_{y, \eta}(x, y)) + \ca{N}_{\Y}(T_{y, \eta}(x, y))  \right)\\
            \leq{}& \mathrm{dist}\left(0,  -\nabla_y f(x, y) + \partial r_2(T_{y, \eta}(x, y)) + \ca{N}_{\Y}(T_{y, \eta}(x, y))  \right) + \eta L_f\norm{R_{y, \eta}(x, y)}\\
            \leq{}& (1+\eta L_f)\norm{R_{y, \eta}(x, y)}  \leq (1+\eta L_f)\varepsilon.
        \end{aligned}
    \end{equation*}
    Here the second inequality follows from the fact that $R_{y, \eta}(x, y) = \frac{1}{\eta}\left(T_{y, \eta}(x, y) - y\right) \in \nabla_y f(x, y) - \partial r_2(T_{y, \eta}(x, y)) - \ca{N}_{\Y}(T_{y, \eta}(x, y))$.

    On the other hand, for the stationarity with respect to the $x$-variable, from \eqref{Eq_Theo_eps_Equivalence_Gamma_0}, we have
    \begin{equation*}
        \mathrm{dist}\left(0,  \nabla_x f(x,y) + \partial r_1(x) +\ca{N}_{\X}(x)  \right) \leq \norm{e_x} + \alpha \eta L_f \norm{R_{y, \eta}(x, y)} \leq (1+ \frac{2L_f}{\mu}) \varepsilon. 
    \end{equation*}
    Therefore, from the Lipschitz continuity of $\nabla_x f$, it holds that 
    \begin{equation*}
        \begin{aligned}
            &\mathrm{dist}\left(0,  \nabla_x f(x,T_{y, \eta}(x, y)) +\partial r_1(x)+\ca{N}_{\X}(x)  \right) \\
            \leq{}& \mathrm{dist}\left(0,  \nabla_x f(x,y)+ \partial r_1(x)+\ca{N}_{\X}(x)  \right) + \eta L_f \norm{R_{y, \eta}(x, y)} \leq (1+ \frac{2L_f}{\mu} + \eta L_f) \varepsilon.
        \end{aligned}
    \end{equation*}
    This completes the proof. 
\end{proof}

In the rest of this subsection, we investigate the boundedness of $\Gamma_{(\eta, \alpha)}$ over $\X \times \Y$. 
Under Assumption \ref{Assumption_f_Minmax},  we denote 
\begin{equation}
    F(x) := \max_{y \in \Y} f(x, y) + r_1(x) - r_2(y).
\end{equation}
The strong concavity of $f(x, \cdot)$ in Assumption \ref{Assumption_f_Minmax}(2) and \cite[Proposition 2.143]{bonnans2013perturbation} imply that $F(x)$ is well-defined and continuous over $\X$. The following proposition shows that $\Gamma_{(\eta, \alpha)}$ in \eqref{Prob_Pen} is bounded below over $\X \times \Y$ under mild conditions. 
\begin{prop}
    \label{Prop_BoundBelow_Gamma}
    Suppose Assumption \ref{Assumption_f_Minmax} holds, $\eta > 0$, and $\alpha \geq \max\{1, \frac{2}{\eta \mu}\}$. Then we have
    \begin{equation}
         \inf_{x \in \X, y \in \Y} \Gamma_{(\eta, \alpha)}(x, y) = \inf_{x \in \X} F(x). 
    \end{equation}
\end{prop}
\begin{proof}
    For any $x \in \X$, let $y^* = \argmin_{y \in \Y} \Gamma_{(\eta, \alpha)}(x, y)$. It holds that 
    \begin{equation} \label{eq:4.24}
        0 \in \left( I_p + \alpha \eta \nabla_{yy}^2 f(x, y^* )  \right)R_{y, \eta}(x, y^* ) + (\alpha-1) \left(R_{y, \eta}(x, y^* ) 
        - \nabla_y f(x, y^* ) + \partial r_2(y^*)\right) + \ca{N}_{\Y}(y^* ).
    \end{equation}
    Then together with Lemma \ref{Le_gradient_inequality_FBEY}, there exists $d_n \in \ca{N}_{\Y}(y^*)$ and $d_y \in \left(R_{y, \eta}(x, y^*) -\nabla_y f(x, y^*)\right)  + \partial r_2(y^*) + \ca{N}_{\Y}(y^*)$ such that 
    \begin{equation*}
        \begin{aligned}
            0={}&\inner{R_{y, \eta}(x, y^*),  \left( I_p +  \alpha\eta \nabla_{yy}^2 f(x, y^*)  \right)R_{y, \eta}(x, y^*) + (\alpha-1) d_y + d_n}\\
            \leq{}& \inner{R_{y, \eta}(x, y^*),  \left( I_p +  \alpha \eta \nabla_{yy}^2 f(x, y^*)  \right)R_{y, \eta}(x, y^*) } + (\alpha-1)\inner{R_{y, \eta}(x, y^*), d_y}\\
            \leq{}& \inner{R_{y, \eta}(x, y^*),  \left( I_p +  \alpha \eta \nabla_{yy}^2 f(x, y^*)  \right)R_{y, \eta}(x, y^*) }\\
            \leq{}&  - \frac{\mu \alpha \eta}{2} \norm{R_{y, \eta}(x, y^*)}^2 \leq 0. 
        \end{aligned}
    \end{equation*}
    Then we can conclude that $R_{y, \eta}(x, y^*) = 0$. Thus \eqref{eq:4.24} reduces to $0\in - \nabla_y f(x, y^* ) + \partial r_2(y^*) + \ca{N}_{\Y}(y^* )$. This, together with
    the strong concavity of $y \mapsto f(x, y)$ in Assumption \ref{Assumption_f_Minmax}(2), implies that $y^* = \argmax_{y \in \Y} f(x, y) + r_1(x) - r_2(y)$. As a result, we have
    $F(x)= f(x,y^*) + r_1(x) - r_2(y^*) = \Gamma_{(\eta,
    \alpha)}(x,y^*) =\inf_{y\in\Y} \Gamma_{(\eta,\alpha)}(x,y)$. Thus
    \begin{equation*}
        \inf_{x \in \X} F(x) = \inf_{x \in \X}\inf_{y \in \Y} \Gamma_{(\eta, \alpha)}(x, y) = \inf_{x \in \X, y \in \Y} \Gamma_{(\eta, \alpha)}(x, y). 
    \end{equation*}
    This completes the proof. 
    % For any $(x, y) \in \Rn$ that minimizes $\Gamma_{(\eta, \alpha)}$, we can conclude from Theorem \ref{Theo_Equivalence_Gamma} that $(x, y)$ is a first-order minimax point of \eqref{Prob_Minmax}. Therefore, we can conclude that 
    % \begin{equation}
    %     \inf_{x \in \X, y \in \Y} \Gamma_{(\eta, \alpha)}(x, y) \geq 
    % \end{equation}
\end{proof}

\subsection{Subgradient descent-ascent method}

By transforming \eqref{Prob_Minmax} into \eqref{Prob_Pen}, various existing approaches in nonsmooth nonconvex optimization over $\X \times \Y$ can be directly implemented to solve \eqref{Prob_Minmax}, and thus solving \eqref{Prob_Con_Minmax} through \eqref{Prob_Pen_Minmax}. Furthermore, the convergence properties of those approaches, including the global convergence and iteration complexity, directly follow the existing related works.

In this subsection, we go beyond the direct application of existing subgradient methods \cite{davis2020stochastic,castera2021inertial,bolte2021conservative,bolte2022long,hu2022constraint,xiao2023adam} to demonstrate that our equivalent minimization problem \eqref{Prob_Pen} allows us to design and analyze the following subgradient descent-ascent method for solving \eqref{Prob_Minmax}:
\begin{equation}
    \label{Eq_SubGDA}
    \left\{
    \begin{aligned}
        &\xkp \in \Pi_{\X}(\xk - \eta_{x,k} (\nabla_x f(\xk, \yk) + \partial r_1(\xk) ) ),\\
        &\ykp = \yk + \eta_{y,k}  R_{y, \eta}(\xkp, \yk).
    \end{aligned}
    \right.
\end{equation}
In particular, we can prove that $\Gamma_{(\eta, \alpha)}$ serves as a Lyapunov function for the differential inclusion corresponding to the discrete update scheme \eqref{Eq_SubGDA}. This further emphasizes the importance of \eqref{Prob_Pen} in proving the convergence properties of the optimization methods designed for solving \eqref{Prob_Minmax}, and thus solving \eqref{Prob_Con_Minmax} through \eqref{Prob_Pen_Minmax}.

Let $\mathcal{C}_h$ denote the collection of all first-order minimax points of \eqref{Prob_Minmax}. That is, 
\begin{equation}
    \ca{C}_h:= \{(x, y) \in \X \times \Y: (0,0) \in \partial h(x, y) + (\ca{N}_{\X}(x), - \ca{N}_{\Y}(y))  \}.
\end{equation}
We begin with the following assumptions on \eqref{Prob_Minmax} and \eqref{Eq_SubGDA}. 
\begin{assumpt}
    \label{Assumption_alg_f}
    \begin{enumerate}
        \item Function $r_1$ is path-differentiable \cite{bolte2021conservative}. 
        \item Subeet $\X$ is a closed convex subset of $\Rn$. 
        \item Subset $\{h(x, y): (x, y) \in \ca{C}_h \}$ is a finite subset of $\bb{R}$. 
    \end{enumerate}
\end{assumpt}

Here we make some comments on Assumption \ref{Assumption_alg_f}. Assumption \ref{Assumption_alg_f}(1) assumes the path-differentiability \cite[Definition 4]{bolte2021conservative} of the term $r_1$. It is worth mentioning that the class of path-differentiable functions is general enough to cover the objectives in a wide range of real-world problems. As shown in \cite[Section 5.1]{davis2020stochastic}, any Clarke regular function is path-differentiable. Beyond Clarke regular functions, another important class of path-differentiable functions are functions whose graphs are definable in an $o$-minimal structure \cite[Definition 5.10]{davis2020stochastic}. Since definability is preserved under finite summation and composition, \cite{davis2020stochastic,bolte2021conservative} shows that numerous common objective functions are all definable. Moreover, Assumption \ref{Assumption_alg_f}(2) is a common assumptions for \eqref{Prob_Minmax}. Furthermore, Assumption \ref{Assumption_alg_f}(3) is referred to as the nonsmooth Morse-Sard property \cite{davis2020stochastic}. As demonstrated in \cite{bolte2021conservative} and Theorem \ref{Theo_Equivalence_Gamma}, Assumption \ref{Assumption_alg_f}(3) holds whenever $h$ is definable, hence is mild in practice. 

Moreover, for the subgradient descent-ascent method \eqref{Eq_SubGDA}, we make the following assumptions.
\begin{assumpt}
    \label{Assumption_alg_f_addition}
    \begin{enumerate}
        \item The sequences of stepsizes $\{\eta_{x,k}\}$ and $\{\eta_{y,k}\}$ are positive and satisfy
        \begin{equation}
            \sum_{k\geq 0} \eta_{x,k} = +\infty,  \quad \sum_{k\geq 0} \eta_{y,k} = +\infty, \quad \text{and} \quad \sup_{k\geq 0} \eta_{y,k} \leq \eta.   
        \end{equation}
        \item There exists $\theta \geq \frac{\alpha \eta L_f^2}{\mu}$ for a prefixed $\alpha \geq \max\{1, \frac{2}{\eta \mu}\}$, such that $\lim_{k\to +\infty} \frac{\eta_{y,k}}{\eta_{x, k}} = \theta$.
        \item For any $\alpha \geq \max\{1, \frac{2}{\eta \mu}\}$, the function $\Gamma_{(\eta, \alpha)}$ is coercive over $\X \times \Y$. 
    \end{enumerate}
\end{assumpt}

In the rest of this subsection, we aim to establish the convergence of \eqref{Eq_SubGDA} under Assumption \ref{Assumption_alg_f}. We begin our analysis with the following definition of differential inclusion. 
\begin{defin}
	\label{Defin_DI}
	For any locally bounded set-valued mapping $\ca{H}: \Rn \rightrightarrows \Rn$ that is nonempty convex-valued and graph-closed, we say that the absolutely continuous mapping $\gamma:\bb{R}_+ \to \Rn$ is a solution for the differential inclusion 
	\begin{equation}
		\label{Eq_def_DI}
		\frac{\mathrm{d} x}{\mathrm{d}t} \in -\ca{H}(x),
	\end{equation}
	with initial point $x_0$, if $\gamma(0) = x_0$ and $\dot{\gamma}(t) \in -\ca{H}(\gamma(t))$ holds for almost every $t\geq 0$. 
\end{defin}

Based on Definition \ref{Defin_DI}, we define $\kappa(x, y) := 2\norm{\partial r_1(x) + \nabla_x f(x, y)}$. Then it is easy to verify that the set-valued mapping $(x, y) \mapsto \ca{B}_{\kappa(x, y)}$ is graph-closed and locally bounded over $\Rn \times \Rp$. Then we consider the following differential inclusion associated with \eqref{Eq_SubGDA},
\begin{equation}
    \label{Eq_SubGDA_DI}
    \left( \frac{\mathrm{d}x}{\mathrm{d}t}, \frac{\mathrm{d}y}{\mathrm{d}t} \right) \in -\left(\partial r_1(x) + \nabla_x f(x, y) + \left(\ca{N}_{\X}(x) \cap \ca{B}_{\kappa(x, y)}\right), - \theta R_{y, \eta}(x, y)  \right).
\end{equation}

We now introduce the concept of the Lyapunov function for the differential inclusion \eqref{Eq_def_DI} with a stable set $\A$. 
\begin{defin}
	\label{Defin_Lyapunov_function}
	Let $\A \subset \Rn$ be a closed set. A continuous function $\Psi:\Rn \to \bb{R}$ is referred to as a Lyapunov function for the differential inclusion \eqref{Eq_def_DI}, with the stable set $\A$, if it satisfies the following conditions:
	\begin{itemize}
		\item for any $\gamma$ that is a solution for \eqref{Eq_def_DI} with $\gamma(0) \notin \A$, it holds that $\Psi(\gamma(t)) < \Psi(\gamma(0))$ for any $t>0$;
		\item for any $\gamma$ that is a solution for \eqref{Eq_def_DI} with $\gamma(0) \in \A$, it holds that $\Psi(\gamma(t)) \leq \Psi(\gamma(0))$ for any $t\geq0$.
	\end{itemize}
\end{defin}

Then the following proposition illustrates that the differential inclusion \eqref{Eq_SubGDA_DI} admits $\Gamma_{(\eta, \alpha)}$ as its Lyapunov function with stable set $\ca{C}_h$. 
\begin{prop}
    \label{Prop_Lyapunov_SubGDA}
    Suppose Assumption \ref{Assumption_f_Minmax}, Assumption \ref{Assumption_alg_f} and Assumption \ref{Assumption_alg_f_addition} hold. Then $\Gamma_{(\eta_y, \alpha)}$ is a Lyapunov function for the differential inclusion \eqref{Eq_SubGDA_DI} with stable set $\ca{C}_h$. 
\end{prop}
\begin{proof}
    For any trajectory $(x(t), y(t))$ of the differential inclusion \eqref{Eq_SubGDA_DI}, there exists $l_x$, $l_n$ such that 
    $\dot{x}(t) = -\big(l_x(t) + \nabla_x f(x(t), y(t)) + l_n(t)\big)$ and $l_x(t) \in \partial r_1(x(t))$, $l_n(t) \in \ca{N}_{\X}(x(t))$ for almost every $t\geq 0$. 
    Then we have 
    \begin{equation}
        \label{Eq_Prop_Lyapunov_SubGDA_0}
        \begin{aligned}
            &\inner{ \partial \Gamma_{(\eta, \alpha)}(x(t), y(t)), (\dot{x}(t), \dot{y}(t)) }\\
            ={}& \inner{ \partial \Gamma_{(\eta, \alpha)}(x(t), y(t)), ( -l_x(t) - \nabla_x f(x(t),\; y(t)) - l_n(t),\; \theta R_{y, \eta}(x(t), y(t)) ) }\\
            ={}& -\inner{\partial r_1(x(t)) + \nabla_x \Xi_{(\eta, \alpha)}(x(t), y(t)),\;  l_x(t) + \nabla_x f(x(t), y(t)) + l_n(t)  } \\
            &+  \theta \inner{\nabla_y \Xi_{(\eta, \alpha)}(x(t), y(t)) -(\alpha-1) \partial r_2(y(t)),\; R_{y, \eta}(x(t), y(t)) }\\
            ={}& -\inner{\partial r_1(x(t)) + \nabla_x \Xi_{(\eta, \alpha)}(x(t), y(t)) + l_n(t) ,\;  l_x(t) + \nabla_x f(x(t), y(t)) + l_n(t)  } \\
            &+  \theta \inner{\nabla_y \Xi_{(\eta, \alpha)}(x(t), y(t)) -(\alpha-1) \partial r_2(y(t)),\; R_{y, \eta}(x(t), y(t)) }.
        \end{aligned}
    \end{equation}
    Here the second equality directly follows from the formulation of $\Gamma_{(\eta, \alpha)}$ and $\Xi_{(\eta, \alpha)}$ in \eqref{Prob_Pen}  and \eqref{Eq_defin_Xi}, respectively. Moreover, the third inequality follows from the convexity of $\X$, which implies that $\inner{l_n(t), \dot{x}(t)} = 0$ for almost every $t\geq 0$, as demonstrated in \cite[Theorem 6.2]{davis2020stochastic}. Therefore, with $\alpha = 2\max\{1, \frac{1}{\eta \mu}\}$ and $\theta \geq \frac{\alpha \eta L_f^2}{\mu}$, it follows from Lemma \ref{Le_gradient_inequality_FBEY} that 
    \begin{equation}
        \label{Eq_Prop_Lyapunov_SubGDA_1}
        \begin{aligned}
            &\inner{R_{y, \eta}(x(t), y(t)), \; \nabla_y \Xi_{(\eta, \alpha)}(x(t), y(t)) -(\alpha-1) \partial r_2(y(t)) } \\
            ={}& \Big\langle R_{y, \eta}(x(t), y(t)), \; (I_p + \alpha \eta \nabla_{yy}^2 f(x(t), y(t)) )  R_{y, \eta}(x(t), y(t))  \\
            &\qquad \qquad \qquad \qquad  -  (\alpha - 1) (\nabla_y f(x(t), y(t)) - R_{y, \eta}(x(t), y(t)) -\partial r_2(y(t)) ) \Big\rangle\\
            \leq{}& \inner{ R_{y, \eta}(x(t), y(t)),\;  (I_p + \alpha \eta \nabla_{yy}^2 f(x(t), y(t)) )  R_{y, \eta}(x(t), y(t)) }\\
            \leq{}& - \mu \alpha \eta \norm{R_{y, \eta}(x(t), y(t))}^2. 
        \end{aligned}
    \end{equation}
    Then by combing \eqref{Eq_Prop_Lyapunov_SubGDA_0} with \eqref{Eq_Prop_Lyapunov_SubGDA_1}, we have
    \begin{equation*}
        \begin{aligned}
            &\inf \inner{ \partial \Gamma_{(\eta, \alpha)}(x(t), y(t)), (\dot{x}(t), \dot{y}(t)) }\\
            \leq{}& - \norm{l_x(t) + \nabla_x f(x(t), y(t)) + l_n(t)}^2  + \alpha \eta L_f \norm{R_{y,\; \eta}(x(t), y(t))} \norm{l_x(t) + \nabla_x f(x(t), y(t)) + l_n(t)}\\
            & - \alpha \theta \eta \mu \norm{R_{y, \eta}(x(t), y(t))}^2\\
            \leq{}& - \frac{1}{2}\norm{l_x(t) + \nabla_x f(x(t), y(t)) + l_n(t)}^2 - \frac{\alpha \theta \mu\eta}{2} \norm{R_{y, \eta}(x(t), y(t))}^2. 
        \end{aligned}
    \end{equation*}
    Here the second inequality directly follows from the choice of $\theta$ that guarantees $\alpha^2 \eta^2 L_f \leq \alpha \theta \mu \eta$. Together with the path-differentiability of $\Gamma_{(\eta, \alpha)}$, we have
    \begin{equation*}
        \begin{aligned}
            &\Gamma_{(\eta, \alpha)}(x(s), y(s)) - \Gamma_{(\eta, \alpha)}(x(0), y(0)) \\
            \leq{}& - \int_{0}^t  \frac{1}{2}\norm{l_x(t) + \nabla_x f(x(t), y(t)) + l_n(t)}^2 + \frac{\alpha \theta \mu\eta}{2} \norm{R_{y, \eta}(x(t), y(t))}^2 \mathrm{d}t\\
            \leq{}& - \int_{0}^t  \frac{1}{2} \mathrm{dist}\Big( 0, \partial r_1(x(t)) + \nabla_x f(x(t), y(t)) + \ca{N}_{\X}(x(t)) \Big)^2 + \frac{\alpha \theta \mu\eta}{2} \norm{R_{y, \eta}(x(t), y(t))}^2 \mathrm{d}t.
        \end{aligned}
    \end{equation*}

    Therefore, for any trajectory that satisfies $(x(0), y(0)) \notin \ca{C}_h$, we can conclude from the continuity of $(x(t), y(t))$ that there exists $t_0> 0$ such that $(x(t), y(t)) \notin \ca{C}_h$ for all $t\in [0, t_0]$. Then for any $t \in [0, t_0]$, it holds that either $0 \notin \partial r_1(x(0)) + \nabla_x f(x(0), y(0)) + \ca{N}_{\X}(x(0))$ or $0\neq R_{y, \eta}(x(0), y(0))$. As a result, for any $s > 0$, it holds that 
    \begin{equation*}
        \begin{aligned}
            &\Gamma_{(\eta, \alpha)}(x(s), y(s)) - \Gamma_{(\eta, \alpha)}(x(0), y(0)) \\
            \leq{}& - \int_{0}^{\min\{t_0, s\}}  \frac{1}{2} \mathrm{dist}\Big( 0, \partial r_1(x(t)) + \nabla_x f(x(t), y(t)) + \ca{N}_{\X}(x(t)) \Big)^2 + \frac{\alpha \theta \mu\eta}{2} \norm{R_{y, \eta}(x(t), y(t))}^2 \mathrm{d}t\\
            <{}& 0.
        \end{aligned}
    \end{equation*}
    Therefore, we can conclude that $\Gamma_{(\eta, \alpha)}$ is a Lyapunov function for the differential inclusion \eqref{Eq_SubGDA_DI} with stable set $\ca{C}_h$. This completes the proof. 
\end{proof}

Proposition \ref{Prop_Lyapunov_SubGDA} illustrates that the subgradient descent-ascent method \eqref{Eq_SubGDA} can be regarded as an inexact subgradient descent method for minimizing $\Gamma_{(\eta, \alpha)}$ over $\X \times \Y$. Therefore, we can directly employ the ODE approaches \cite{benaim2005stochastic,borkar2009stochastic,duchi2018stochastic,davis2020stochastic,josz2023lyapunov,xiao2023convergence} to establish the convergence of \eqref{Eq_SubGDA}. 
Based on the results in \cite[Theorem 3.5]{xiao2023convergence}, we
can establish the following theorem illustrating the convergence of \eqref{Eq_SubGDA}.
\begin{theo}
    \label{Theo_convergence_SubSGD}
    Suppose Assumption \ref{Assumption_f_Minmax}, Assumption \ref{Assumption_alg_f} and Assumption \ref{Assumption_alg_f_addition} hold. Then for any given $\varepsilon > 0$, there exists $\varepsilon_{ub}> 0$ such that for any sequences $\{\eta_{x, k}\}$ and $\{\eta_{y,k}\}$ that satisfy $\sup_{k\geq 0} \max\{\eta_{x,k}, \eta_{y,k}\} \leq \varepsilon_{ub}$, for the sequence $\{(\xk, \yk)\}$ generated by \eqref{Eq_SubGDA}, it holds that 
    \begin{equation}
        \limsup_{k\to +\infty}~ \mathrm{dist}\left( (\xk, \yk), \ca{C}_h \right) \leq \varepsilon. 
    \end{equation}
\end{theo}
\begin{proof}
    We first verify the validity of the conditions in \cite[Theorem 3.5]{xiao2023convergence}. The set-value mapping $(x, y) \mapsto \left(\partial r_1(x) + \nabla_x f(x, y) + \left(\ca{N}_{\X}(x) \cap \ca{B}_{\kappa(x, y)}\right), - \theta R_{y, \eta_y}(x, y)  \right)$ is convex-valued, locally bounded and graph-closed. In addition, the differential inclusion \eqref{Eq_SubGDA_DI} admits $\Gamma_{(\eta, \alpha)}$ as its Lyapunov function. Then  together with Assumption \ref{Assumption_alg_f}(3), we can conclude that the differential inclusion \eqref{Eq_SubGDA_DI} satisfies \cite[Assumption 3.3]{xiao2023convergence}. Therefore, based on the results in \cite[Theorem 3.5]{xiao2023convergence}, we can conclude that $\limsup_{k\to +\infty}~ \mathrm{dist}\left( (\xk, \yk), \ca{C}_h \right) \leq \varepsilon$. 
    This completes the proof.
\end{proof}

\section{Numerical Experiments}
In this section, we perform preliminary numerical experiments to evaluate the efficiency of our proposed approach to solve adversarial attacks in resource allocation problems \cite{tsaknakis2023minimax}. All numerical experiments in this section are conducted on a server equipped with an Intel i7-1260P CPU, running Python 3.8, PyTorch 2.3.1, NumPy 1.16.4, and SciPy 1.11.4.

\begin{table}[tb]
\centering
\begin{tabular}{cc|ccccc}
\hline
                                           &          & fval & iter & stat & feas & time \\ \hline
\multicolumn{1}{l|}{\multirow{3}{*}{$n, p=10$}}& L-BFGS-B &-6.07e+00   & 77  & 1.64e-08  & 1.67e-07  & 0.15\\
\multicolumn{1}{l|}{}                      & TNC      &-6.07e+00   & 30  & 4.42e-09  & 2.11e-08  & 0.21\\
\multicolumn{1}{l|}{}                      & GDA      &-6.07e+00   & 4361  & 9.98e-08  & 0.00e+00  & 5.70\\ \hline
\multicolumn{1}{l|}{\multirow{3}{*}{$n, p=20$}} & L-BFGS-B &-1.11e+01   & 94  & 1.50e-08  & 1.12e-07  & 0.16\\
\multicolumn{1}{l|}{}                      & TNC      &-1.11e+01   & 47  & 2.83e-09  & 3.71e-08  & 0.33\\
\multicolumn{1}{l|}{}                      & GDA      &-1.11e+01   & 3325  & 9.99e-08  & 5.39e-08  & 4.47\\ \hline       
\multicolumn{1}{l|}{\multirow{3}{*}{$n, p=50$}} & L-BFGS-B &-2.69e+01   & 292  & 3.88e-08  & 2.25e-07  & 0.48\\
\multicolumn{1}{l|}{}                      & TNC      &-2.69e+01   & 83  & 4.03e-09  & 3.72e-08  & 0.75\\
\multicolumn{1}{l|}{}                      & GDA      &-2.69e+01   & 10000  & 2.92e-05  & 2.13e-05  & 15.78\\ \hline
\multicolumn{1}{l|}{\multirow{3}{*}{$n, p=100$}} & L-BFGS-B &-5.20e+01   & 319  & 2.98e-08  & 3.51e-07  & 0.58\\
\multicolumn{1}{l|}{}                      & TNC      &-5.20e+01   & 143  & 7.25e-09  & 8.21e-08  & 1.16\\
\multicolumn{1}{l|}{}                      & GDA      &-5.19e+01   & 10000  & 1.51e-06  & 6.49e-07  & 13.91\\ \hline
\multicolumn{1}{l|}{\multirow{3}{*}{$n, p=200$}} & L-BFGS-B &-1.04e+02   & 458  & 2.71e-08  & 3.65e-07  & 1.49\\
\multicolumn{1}{l|}{}                      & TNC      &-1.04e+02   & 268  & 3.47e-09  & 1.63e-08  & 4.32\\
\multicolumn{1}{l|}{}                      & GDA      &-1.04e+02   & 10000  & 1.57e-05  & 6.02e-06  & 37.58\\ \hline
\multicolumn{1}{l|}{\multirow{3}{*}{$n, p=500$}} & L-BFGS-B &-2.56e+02   & 773  & 2.79e-08  & 5.54e-07  & 3.65\\
\multicolumn{1}{l|}{}                      & TNC      &-2.56e+02   & 679  & 4.07e-09  & 4.07e-08  & 20.71\\
\multicolumn{1}{l|}{}                      & GDA      &-2.56e+02   & 10000  & 2.93e-05  & 9.94e-06  & 33.96\\ \hline
\multicolumn{1}{l|}{\multirow{3}{*}{$n, p=1000$}} & L-BFGS-B &-5.08e+02   & 1129  & 3.09e-08  & 2.95e-07  & 8.72\\
\multicolumn{1}{l|}{}                      & TNC      &-5.08e+02   & 1274  & 3.70e-09  & 4.39e-08  & 33.32\\
\multicolumn{1}{l|}{}                      & GDA      &-5.08e+02   & 10000  & 3.75e-05  & 5.80e-06  & 44.55\\ \hline
\end{tabular}
\caption{Numerical results for solving \eqref{Eq_test_synthetic} with fixed $c = 1$. }
\label{Table_num1_fixc}
\end{table}

\begin{table}[tb]
\centering
\begin{tabular}{cc|ccccc}
\hline
                                           &          & fval & iter & stat & feas & time \\ \hline
\multicolumn{1}{l|}{\multirow{3}{*}{$c=0.1$}}& L-BFGS-B &-6.71e+01   & 200  & 6.89e-09  & 2.71e-06  & 0.38 \\
\multicolumn{1}{l|}{}                      & TNC      &-6.31e+01   & 301  & 5.55e-09  & 1.57e-06  & 3.49\\
\multicolumn{1}{l|}{}                      & GDA      &-6.65e+01   & 6043  & 9.78e-08  & 9.88e-07  & 7.94\\ \hline
\multicolumn{1}{l|}{\multirow{3}{*}{$c=0.2$}} & L-BFGS-B &-6.37e+01   & 125  & 2.63e-09  & 7.72e-08  & 0.27\\
\multicolumn{1}{l|}{}                      & TNC      &-5.95e+01   & 233  & 1.58e-09  & 1.75e-07  & 2.22\\
\multicolumn{1}{l|}{}                      & GDA      &-6.30e+01   & 6480  & 9.74e-08  & 2.04e-08  & 8.24\\ \hline             
\multicolumn{1}{l|}{\multirow{3}{*}{$c=0.5$}} & L-BFGS-B &-5.60e+01   & 178  & 4.99e-09  & 1.12e-07  & 0.34\\
\multicolumn{1}{l|}{}                      & TNC      &-5.53e+01   & 218  & 1.89e-09  & 1.10e-08  & 1.77\\
\multicolumn{1}{l|}{}                      & GDA      &-5.56e+01   & 10000  & 7.47e-06  & 1.87e-05  & 13.50\\ \hline
\multicolumn{1}{l|}{\multirow{3}{*}{$c=0.7$}} & L-BFGS-B &-5.33e+01   & 268  & 7.21e-09  & 7.27e-08  & 0.49\\
\multicolumn{1}{l|}{}                      & TNC      &-5.33e+01   & 215  & 2.42e-09  & 6.33e-08  & 1.85\\
\multicolumn{1}{l|}{}                      & GDA      &-5.31e+01   & 10000  & 1.36e-06  & 1.40e-06  & 15.85\\ \hline
\multicolumn{1}{l|}{\multirow{3}{*}{$c=1.0$}} & L-BFGS-B &-5.20e+01   & 319  & 2.98e-08  & 3.51e-07  & 0.58\\
\multicolumn{1}{l|}{}                      & TNC      &-5.20e+01   & 143  & 7.25e-09  & 8.21e-08  & 1.16\\
\multicolumn{1}{l|}{}                      & GDA      &-5.19e+01   & 10000  & 1.51e-06  & 6.49e-07  & 13.91\\ \hline
\multicolumn{1}{l|}{\multirow{3}{*}{$c=2.0$}} & L-BFGS-B &-5.22e+01   & 386  & 8.60e-09  & 0.00e+00  & 0.65\\
\multicolumn{1}{l|}{}                      & TNC      &-5.22e+01   & 134  & 1.49e-09  & 0.00e+00  & 1.19\\
\multicolumn{1}{l|}{}                      & GDA      &-5.21e+01   & 10000  & 1.37e-06  & 0.00e+00  & 13.75\\ \hline
\multicolumn{1}{l|}{\multirow{3}{*}{$c=5.0$}} & L-BFGS-B &-5.22e+01   & 347  & 6.49e-09  & 0.00e+00  & 0.64\\
\multicolumn{1}{l|}{}                      & TNC      &-5.22e+01   & 127  & 1.13e-09  & 0.00e+00  & 1.13\\
\multicolumn{1}{l|}{}                      & GDA      &-5.22e+01   & 8292  & 1.00e-07  & 0.00e+00  & 19.01\\ \hline
\end{tabular}
\caption{Numerical results for solving \eqref{Eq_test_synthetic}
with fixed $(n,p) = (100,100)$. }
\label{Table_num1_fixnp}
\end{table}

We evaluate the efficiency of existing solvers for minimization problems on solving the following synthetic test problem,
\begin{equation}
    \label{Eq_test_synthetic}
    \min_{0\leq x\leq 1} \left(\max_{y \in \Rp, ~x+ y\leq c} \inner{b, x} +  \inner{x, By} - \frac{1}{2} \norm{y}^2\right). 
\end{equation}
Here, the matrix $B \in \bb{R}^{n\times p}$ and the vector $b \in \Rn$ are randomly generated by the ``randn'' function in the NumPy package. Then we project $b$ to the unit sphere in $\Rn$. Moreover, $c > 0$ is a prefixed constant. It is easy to verify that the optimization problem \eqref{Eq_test_synthetic} satisfies all the assumptions in Assumption \ref{Assumption_f}. When transferring \eqref{Eq_test_synthetic} into the form of minimization problem \eqref{Prob_Pen}, we choose $\eta = 1$ and $\alpha = 100$ for all test instances.

In order to evaluate the efficiency of existing solvers for minimization problems over $\X \times \Y$, we choose the ``L-BFGS-B'' and ``TNC'' solvers from the SciPy package to solve the minimization problem \eqref{Prob_Pen}. The L-BFGS-B solver uses the limited memory BFGS algorithm \cite{byrd1995limited,zhu1997algorithm} to solve optimization problems with bounded constraints and is originally developed in FORTRAN language \cite{zhu1997algorithm}. We call the L-BFGS-B solver in our numerical experiments by its Python interface through the SciPy package, where we set ``gtol'' as $10^{-7}$ and keep all the other parameters as their default values. 
Moreover, the TNC solver uses a truncated Newton method \cite{nash1984newton,nocedal1999numerical} to solve optimization problems with bounded constraints, where the subproblems are solved by the truncated conjugate gradient method \cite{yuan2000truncated}. The TNC solver is implemented in C language, and we use its Python interface through the SciPy package. In all the test instances, we set ``gtol = $10^{-7}$'' and ``maxfun = $10000$'' while keeping all the other parameters as their default values. Furthermore, we also test the performance of the proximal gradient descent-ascent (GDA) method \cite{xu2023unified} for solving \eqref{Prob_Minmax}.  For the GDA method, we choose the best stepsizes for minimization and maximization parts from $\{a_1 \times 10^{-a_2}: a_1 \in \{ 1,3,5,7,9\}, ~ a_2 \in \{1, 2, 3, 4\} \}$ in each test instance. Then we terminate the GDA solver when $\norm{ \mathrm{Proj}_{\X \times \Y}\left( (\xk, \yk) - \nabla \Gamma_{(\eta, \alpha)}(\xk, \yk) \right)  - (\xk, \yk)} \leq 10^{-7}$.  In addition, it is worth mentioning that under our settings, all the tested solvers only utilize the first-order derivatives of the objective function and constraints. Moreover, we employ the automatic differentiation package from PyTorch \cite{paszke2017automatic} to automatically compute the gradients of $h$ and $\Gamma_{(\eta, \alpha)}$.

In the numerical results, we report the function value obtained, as well as the feasibility and stationarity of the result $(\tilde{x}, \tilde{y})$ obtained by the tested solvers. Here the feasibility is measured by $\norm{\max\{ \tilde{x} + \tilde{y} - c, 0 \}}$, and the stationarity is measured by $\frac{\norm{ \mathrm{Proj}_{\X \times \Y}\left( (\tilde{x}, \tilde{y}) - \nabla \Gamma_{(\eta, \alpha)}(\tilde{x}, \tilde{y}) \right)  - (\tilde{x}, \tilde{y})} }{\norm{\nabla \Gamma_{(\eta, \alpha)}(x_0, y_0)}}$. 
Table \ref{Table_num1_fixc} and Table \ref{Table_num1_fixnp} present the results of our numerical experiments. Here ``fval'', ``iter'', ``stat'', ``feas'', and ``time'' refer to the function value of $\Gamma_{(\eta, \alpha)}$ at the final solutions, the total iterations taken by the solvers, the stationarity of the final solutions, the feasibility, and the wall-clock time taken by the solvers, respectively. From these tables, we can conclude that directly applying L-BFGS-B and TNC solvers to minimize $\Gamma_{(\eta, \alpha)}$ achieves superior performance to the GDA method for solving the minimax problem \eqref{Prob_Minmax} in all the test instances. Moreover, we notice that the L-BFGS-B solver is faster than the TNC solver while achieving similar stationarity in most test instances. Therefore, we can conclude that \eqref{Prob_Pen} enables the direct implementations of existing solvers for minimization  problems
to efficiently solve the minimax problem  \eqref{Prob_Con_Minmax} through \eqref{Prob_Pen}.

\section{Conclusion}
In this paper, we study a class of nonconvex-strongly-concave minimax 
problem \eqref{Prob_Con_Minmax}, where the constraints of the maximization part depend on the variables of the minimization part. We first prove that under Assumption \ref{Assumption_f}, \eqref{Prob_Con_Minmax} and \eqref{Prob_Pen_Minmax} have the same first-order, local and global minimax points. Therefore, directly applying the descent-ascent methods to solve \eqref{Prob_Pen_Minmax} yields first-order minimax points of \eqref{Prob_Con_Minmax}. Then we focus on the nonconvex-strongly-concave 
minimax problem \eqref{Prob_Minmax} over $\X \times \Y$, which is a generalization of \eqref{Prob_Pen_Minmax}. Motivated by forward-backward envelope for minimization problems, we consider the PFBE for the minimax 
problem \eqref{Prob_Minmax} and present its basic properties. Moreover, by performing PFBE to \eqref{Prob_Minmax}, we propose the minimization problem \eqref{Prob_Pen} and prove that the first-order stationary points of \eqref{Prob_Pen} coincide with the first-order minimax points of \eqref{Prob_Minmax}.

Based on the equivalence between \eqref{Prob_Minmax} and \eqref{Prob_Pen}, various existing optimization approaches for minimization problems over $\X \times \Y$ can be directly implemented to solve \eqref{Prob_Minmax}. Consequently, we can also directly apply these optimization approaches for minimization problems to solve \eqref{Prob_Con_Minmax} through \eqref{Prob_Pen_Minmax}. In particular, we show the convergence of the subgradient descent-ascent method under nonsmooth nonconvex settings, where $\Gamma_{(\eta, \alpha)}$ in \eqref{Prob_Pen} serves as the Lyapunov function of the associated differential inclusion \cite{xiao2023convergence}. Moreover, we perform preliminary numerical experiments, where some existing efficient numerical solvers, including L-BFGS-B and TNC, are applied to solve \eqref{Prob_Con_Minmax} through \eqref{Prob_Pen}. The numerical results demonstrate that applying these solvers to \eqref{Prob_Pen} can achieve superior performance in efficiency over applying
the gradient descent-ascent method to the minimax problem \eqref{Prob_Minmax}. These results demonstrate the promising potential of our proposed minimization approach for solving \eqref{Prob_Con_Minmax} through \eqref{Prob_Pen}.

\section*{Acknowledgement}
 The research of Xiaoyin Hu is supported by Zhejiang Provincial Natural Science Foundation of China under Grant (No. LQ23A010002), the National Natural Science Foundation of China (Grant No. 12301408), and the advanced computing resources provided by the Supercomputing Center of HZCU.

\bibliographystyle{plain}
\bibliography{ref}

\end{document}